\documentclass[amsart]{article}
\usepackage{mitcommands}
\usepackage{tikz}
\usepackage{tikz-cd}
\newcommand{\uX}{\underline{X}}
\newcommand{\uU}{\underline{U}}
\newcommand{\uV}{\underline{V}}
\newcommand{\uW}{\underline{W}}
\newcommand{\uY}{\underline{Y}}
\newcommand{\uG}{\underline{G}}
\newcommand{\uK}{\underline{K}}
\newcommand{\Det}{\operatorname{Det}}
\newcommand{\sDet}{\operatorname{sDet}}
\newcommand{\Ber}{\operatorname{Ber}}
\newcommand{\g}{\mathfrak{g}}
\newcommand{\glie}{\mathfrak{g}}
\newcommand{\tlie}{\mathfrak{t}}

\newcommand{\Loc}{\operatorname{Loc}}
\newcommand{\Pf}{\mathrm{Pf}}
\newcommand{\GQ}{{G_Q}}
\newcommand{\uD}{\underline{D}}
\newtheorem*{clemma}{Lemma}

\date{\today}
\author{Vera Serganova and Dmitry Vaintrob}

\title{Localization for CS manifolds and volume of homogeneous superspaces}

\begin{document}

\maketitle
\abstract{We prove the Schwarz-Zaboronsky localization theorem for CS manifolds and use this to give a volume calculation for homogeneous superspaces for super-Lie groups that lack a real form.}

\newpage
\tableofcontents
\newpage

\section{Introduction}

In the paper \cite{schwarz-zaboronsky}, Schwarz and Zaboronsky prove a powerful localization result in supergeometry which in particular implies the Duistermaat-Heckman formula for hamiltonian actions. Namely, they show that if $X$ is a super-manifold with an odd vector field $Q$, such that $Q^2$ has a certain compactness property and $Q$ has well-behaved fixed points, then any $Q$-equivariant Berezinian integral localizes to an integral over the fixed point submanifold $X^Q.$

Their proof is in the context of (real) super-manifolds; in this paper we re-interpret it slightly in a way that applies to CS super-manifolds in the case of isolated fixed points. 

We apply this theorem to the representation theory of quasireductive algebraic
supergroups, i.e., supergroups whose reduced group is reductive. In
\cite{serganova-sherman} the first author and A. Sherman introduced
the notion of a splitting subgroup. If $K\subset G$ is a pair of
quasireductive groups then $K$  is splitting in $G$ if the trivial
$G$-submodule $\mathbb C\subset \operatorname{Ind}^G_K\mathbb C$
splits as a $G$-module. The splitting condition is equivalent to
relative reductivity: any short exact sequence of $G$-modules splits
if and only if it splits over $K$. The analogous notion is trivial in the theory of compact Lie groups, 
where any inclusion $K\subset G$ is splitting (since the representation theory is semisimple),
but has an analogue in modular representation theory. Namely, working 
over a base field of characteristic $p$, it turns out that an inclusion of finite groups 
$K\subset G$ is splitting if and only if $K$ contain a
Sylow $p$-subgroup of $G$. Many properties of representations of $G$ can be deduced from
those of $K$. Splitting subgroups should be useful for giving a super-Lie group generalisation of 
the Green correspondence theorem from modular representation
theory of finite groups. They are also important in the theory of support
varieties in the super Lie group setting.

In \cite{serganova-sherman} small splitting subgroups
are constructed for $G=GL(m|n)$, $G=Q(n)$ and all Kac-Moody supergroups
of defect $1$. For a general Kac-Moody supergroup $G$, the authors of that paper define a
defect subgroup $D$ of $G$ isomorphic to a direct product of several
copies of $SL(1|1)$ and conjecture that $D$ is minimal splitting. They were
able to prove that $D$ is splitting for $GL(m|n)$ by passing to the
compact real form $X$ of the homogeneous superspace $G/K$ where $K$ is the connected
component of the centralizer of $\uD$. After they show that $X$ has
nonzero volume the result follows by classical unitary trick. For
other Kac-Moody supergroups they encounter an obstacle, as $G/K$ does not
have a compact real form. In the present paper we are able to obtain a simple uniform proof
of the fact that $D$ is splitting for all Kac-Moody supergroups
$G$ using the CS structure on $G/K$ and the new CS localization formula.
A nice corollary of this fact is a "projectivity criterion" for all Kac-Moody
supergroups (see \cite{GHSS} for the history of this notion). Using the same method we
obtain small splitting subgroups of periplectic groups. This fully resolves the question of 
constructing small splitting subgroups (i.e., splitting subgroups whose only simple 
subfactors are abelian supergroups and $SL_2$) for all simple and almost simple 
quasireductive groups $G$ which appear in the Kac classification, \cite{Kac}.




\subsection{Notation and conventions}
In this paper we work with $C^\infty$ CS manifolds, i.e., ringed spaces of $\cc$-super-algebras which are locally isomorphic to $\rr^n$ with ring of functions $C^\infty_\cc(\rr^{m\mid n}).$ All ringed spaces and all instances of the sheaf $C^\infty$ will by convention be taken with complex coefficients. Given a CS manifold $X$, write $\uX$ for the underlying topological space and $X^{red}$ for the purely even supermanifold $\big(\uX, (C^\infty_X)^{red}\big)$ with reduced sheaf of functions, which is canonically isomorphic as a ringed space to the sheaf of complex-valued $C^\infty$ functions on the underlying manifold $\uX.$ 

All algebra in this paper is super, so ``derivation'' means super-derivation, ``commutator'' means super-commutator, etc. All manifolds are super-ringed spaces with a suitable local model.

We write \emph{CS vector space} for the notion of a complex supervector space $V = V_0 \oplus \Pi V_1$ with a choice of real structure $(V_0)_\rr\subset V_0$ on the even part. Associated to such a space, we have a ``linear'' CS manifold $V_{CS}$ with underlying space $(V_0)_\rr$ and sheaf of functions
$$C^\infty(V_{CS}) : = C^\infty((V_0)_\rr)\otimes S^*\Pi(V_1^*).$$ It is clear that any (even) linear map of CS spaces $L:V\to W$ which takes $(V_0)_\rr$ to $(W_0)_\rr$ induces a map $L_{CS}:V_{CS}\to W_{CS}$ in a functorial way. We say that an \emph{orientation} on a CS vector space $V$ is a choice of orientation on the real vector space $(V_0)_\rr,$ equivalently a choice of orientation on the CS manifold $V_{CS}.$

A vector field on a CS supermanifold is a derivation of the sheaf of functions, equivalently a section of the sheaf $T_X$ (defined as usual, as the normal bundle of the diagonal in $X^2$). The reduced part (a.k.a. ``body'') of a vector field $\xi$ is the restriction $\xi\mid_{X^{red}},$ a section of $$T^{red}X:=TX\mid_{X^{red}}.$$ Note that we have canonically $TX\mid_{X^{red}}\cong T(X^{red})\oplus \Pi N_{X/X^{red}},$ where $N_{X/X^{red}}$ is the normal bundle of odd directions. In particular, note that the complex bundle $T_0^{red}: = T^{red}(X)_0$ is isomorphic (as a bundle over $X^{red}$) to the tangent bundle $T(X^{red})$ of the even manifold, and in particular has canonical real structure $T^{red}_{0,\rr}\subset T^{red}_0$ inherited from the canonical real structure on the even manifold $X^{red}.$ The odd part $T^{red}_1\cong N_{X/X^{red}}$ (a bundle on $X^{red}$) will not in general have canonical real structure, and indeed can have obstructions to any choice of real structure (e.g. a nonzero first chern class). Note that an even vector field $\xi$ can be (locally) integrated to positive time if and only if $\xi^{red}$ is real, and this gives an alternative intrinsic characterization of real structure on $T^{red}_0.$

Given a vector field $Q$ (either even or odd), we define the invariant vanishing ideal sheaf $$I_Q(X): = C^\infty\cdot Q(C^\infty)$$ (following the algebraic geometry definition). It is obvious that if $\xi$ is a vector field (either even or odd), then the restriction $I_Q(X)\mid X^{red}$ is precisely the ideal of vanishing of $\xi^{red}.$

Given a super-vector space $V$, respectively, a super-line bundle $V$ over a base CS manifold $X$, its (super-)determinant $\sDet(V)$ is understood in the super-sense as a $(1,0)$-dimensional space, resp., as a line bundle over $X$ (so if $V$ is a vector space, resp., a super-vector bundle over an even base, we have canonically $\sDet(V) = \Det(V_0)\otimes \Det(V_1)^{-1}$). We view $\Det$ as a functor on the category of super-vector vector spaces with isomorhpisms, resp., on the category of super-bundles with isomorhpism. In particular, given an automorphism $V\to V$ of a vector space its determinant give an automorphism $\t{Aut}(\sDet(V))$ of a one-dimensional space, which is canonically identified with $\cc^*.$ We define the Berezinian line bundle $\Ber_X: = \sDet(T^*_X).$ Sections of a Berezinian on an oriented CS manifold can be integrated in the same way (and using the same formulas) as on a (real) supermanifold: see e.g. \cite{witten}.

\subsection{The CS localization theorem}
In the Schwarz-Zaboronsky paper \cite{schwarz-zaboronsky}, they prove a localization result for a manifold with a choice of vector field $Q$ that has a so-called \emph{compactness} condition. This condition is equivalent to the existence of a certain compact super-Lie group $\GQ$ which acts on $X,$ such that $Q$ is the vector field associated to the action of some odd Lie algebra element $Q_{lie}\in \glie_1$ for $\glie$ the Lie algebra associated to $\GQ$. (Note that in the sequel, we will abuse notation and denote both the lie algebra element $Q_{lie}\in \glie_1$ and the vector field $Q$ given by its action by the same letter.) In the CS case we will also work with a group $\GQ$ with a distinguished odd Lie algebra element, which we consisder fixed throughout the paper. 
\begin{defi}
A CS $Q$-group is a \emph{compact} CS super-Lie group $\GQ$ with Lie algebra $\g$ and a choice of odd element $Q\in \g_1,$ such that $\g_0$ is commutative and central and $\g_1$ is one-dimensional and generated by $Q.$ 
\end{defi}
Note that this implies that $\GQ^{red}$ is a real torus, and we will write by convention $$T : = \GQ^{red}.$$

It is clear that a $Q$-group $(\GQ,Q)$ is determined up to isomorphism by the torus $T: = \GQ^{red}$ and the element $Q^2\in \mathfrak{t}_\cc$ in the \emph{complex} Lie algebra of $T.$ 

Given a $Q$-group $\GQ$, we will be studying $\GQ$-equivariant manifolds $X$, with additional $\GQ$-equivariant (or sometimes, only $T$-equivariant) structures.

We will prove a localization theorem under an additional two conditions on the $\GQ$-action, namely:
\begin{defi}
A (reduced) fixed point $x\in \uX^Q$ in a $Q$-manifold is nondegenerate if, for every $Q$-fixed point $x\in \uX^Q,$ the action of $Q$ on the tangent bundle $T_xX$ is an odd automorphism.
\end{defi}
Note that nondegeneracy of a fixed point $x\in \uX$ means that $x$ is a ``scheme-theoretically'' isolated fixed point, i.e., that in an open neighborhood of $x$ we have an equality of sheaves of ideals $I_Q = m_x$ between $I_Q$ and the maximal ideal at the point. Suppose $\GQ$ is a $Q$-group and $X$ is a $\GQ$-manifold with discrete $Q$-fixed points which are nondegenerate. Let $i:X^Q\to X$ be the embedding of the fixed points. In this context we prove the following result.



 \begin{thm}[Localization]\label{thm:main}
 There exists an isomorphism of line bundles on $X^Q$ 
   \[\Loc: i^*\Ber_X\to C_{X^Q}, \] such that,
   for any compactly supported and $Q$-equivariant Berezinian form
   $\omega\in \Gamma(\uX, \Ber_X),$ we have the equality
   \[\int_X\omega = \sum_{X^Q} \Loc\omega. \]
 \end{thm}
 Note that since $X^Q$ is discrete, line bundles on it are simply collections of one-dimensional vector spaces indexed by fixed points (points of $X^Q$). The map $\Loc$ is thus equivalent to choosing a parametrization of the Berezinian fiber $\Ber_X\mid_{x}$ with $\cc$ for each point $x\in X^Q.$ We give an explicit formula for this identification in definition \ref{def:linloc}, and deduce a precise (and computable) version localization formula in Theorem \ref{thm:main_precise}.

This is a generalization to the CS context of one of the main results (Lemma 2) in \cite{schwarz-zaboronsky}. Note that the main result of \cite{schwarz-zaboronsky}, their Lemma 3, implies that an analogous localization result also holds (in the real case) when the fixed points are reduced and nondegenerate but not necessarily discrete. We expect such a generalization to hold also in the CS case, though we do not prove it here.

Along the way we also prove the following \emph{$Q$-acycicity} lemma.
\begin{lm}[$Q$-acyclicity]\label{lm:qex}\label{lm:acyc}
Suppose $X$ is a $\GQ$-equivariant CS manifold and $Q$ is everywhere nonvanishing. Then there exists an odd function $\beta \in C^\infty(X)_1$ such that $Q\beta = 1.$
\end{lm}
This is a CS version of a secondary result, Theorem $1$, of \cite{schwarz-zaboronsky}. Note that this lemma implies the main theorem in the case where $Q$ has no fixed points: indeed, then the right hand side of the localization formula is an empty sum, hence $0$, whereas for any $\omega$ with $Q\omega = 0$ we can write the left hand integral $\int\omega = \int Q(\beta\omega),$ which is the integral of a total derivative of a compactly supported Berezinian --- hence equal to zero.

\subsection{Comparison with Schwarz-Zaboronsky}
Recall that a (real) supermanifold is the same thing as a CS manifold with a choice of real structure $C^\infty_\rr\subset C^\infty$ on functions. Thus any suitably functorial result about CS manifolds implies a corresponding result for real supermanifolds. We can thus define a \emph{real $Q$-group} $\GQ$ to be a real Lie group with a choice of real vector $Q\in \glie_1,$ such that the data $(\GQ, Q)$ forms a CS $Q$-Lie group after forgetting the real structure. 

Since $Q$ is real, $Q^2\in \glie_{0,\rr}$ is real, and so we can define exponents $\exp(t Q^2)\in \GQ^{red} = T$ for any real number $t\in \rr.$ We say a real $Q$-Lie group $(\GQ, Q)$ is \emph{minimal} if $\exp(\rr Q^2)$ is dense in $\GQ^{red}.$ Note that in this case, the action of $\GQ$ on a manifold or a vector space is determined (by continuity) by the action of $Q$, and a function or vector is $\GQ$-invariant if and only if it is $Q$-invariant.

\begin{rmk}\label{rmk:real-minimal} It is, furthermore, easy to see that any real $Q$-Lie group is a product of a minimal $Q$-Lie group with an even torus.
  \end{rmk}

In our language, Schwarz and Zaboronsky prove the following theorem in the case of vector fields with isolated fixed points.
\begin{cthm}[Schwarz-Zaboronsky, \cite{schwarz-zaboronsky}, Section 4]
Suppose given a (real) minimal super $Q$-Lie group $\GQ$ and a $\GQ$-equivariant space $X$, such that $Q$ has isolated and nondegenerate fixed points on $X$. Then the integral $$\int \omega$$ of a compactly supported $\GQ$-equivariant Berezinian $\omega\in \Gamma_{comp}(X, \Ber_X)$ only depends on the vectors $$\omega\mid_{x_i}\in \Ber_X\mid_{x_i},$$ for $x_i$ running through the set of fixed poins $X^Q.$
\end{cthm}
Moreover, \cite{schwarz-zaboronsky} gives a formula for the contribution at each fixed point in terms of local coordinates and an auxiliary choice of a metric on its tangent space. 

Note that the minimality condition can be easily removed here by Remark \ref{rmk:real-minimal}, and the Schwarz-Zaboronsky result holds for arbitrary real $Q$-groups $\GQ$ acting on $X$ with nondegenerate fixed points. The result of the present paper replaces both the manifold $X$ and also the group $\GQ$ by CS versions. Note that (the CS group underlying) a real $Q$-Lie group $\GQ$ can act on a CS manifold, and we already obtain a new result in this case. Our result is also more general in allowing the group $\GQ$ to be itself CS (and the element $Q$ to be non-real). In this context, when $Q^2\in \glie_0$ is not a real even Lie algebra element, there seems to be no good notion of minimality for a Lie group $\GQ$ associated to $Q$ (and certainly it is unreasonable to expect that the action of $\GQ$ is determined by the vector field associated to $Q$). Thus the decision of carrying the full $Q$-Lie group $\GQ$ (rather than just the vector field given by the action of $Q\in \glie_1$) is necessary in this context.


\subsection{Idea of proof of localization}
The key idea of our proof is to refine a statement about scalars to a statement about \emph{distributions}, i.e., possibly singular integral forms on the manifold $X$. Indeed, on the one hand, any Berezinian $\omega$ on $X$ is a non-singular distribution. On the other hand, if $X$ is a CS manifold and $x\in X$ is a reduced point, we have a \emph{delta distribution} $\delta_x,$ which is an even distribution supported at $x$ and uniquely determined by the property
$$\int_X \delta_x\cdot f = f(x)$$
for any even function $f$ on the CS manifold. In particular, since $\int_X\delta_x = 1$ for any $x\in X,$ the localization formula is equivalent to the statement
$$\int_X \omega = \int_X\sum_{x\in X^Q} (\Loc_x \omega)\cdot \delta_x$$ (for $\omega$ a compactly supported and $\GQ$-invariant Berezinian), where $\Loc_x\omega$ are the scalar local contributions. Our proof proceeds by showing that the difference between these two distributions is a total derivative, i.e., we find an odd compactly supported distribution $D$ such tha
$$Q D = \omega - \sum_{x\in X^Q} (\Loc_x \omega)\cdot \delta_x.$$
Since the integral of a total derivative is zero, the localization theorem follows. We construct such a distribution-valued enhancement of the localization result for certain ``nice'' $\GQ$-manifolds (namely linear $\GQ$-manifolds and $\GQ$-manifolds with no $Q$-fixed points), and show that any manifold $X$ with nondegenerate $Q$-fixed points can be glued out of such models.

\

We note that, at first glance, our proof is quite different from Schwarz and Zaboronsky's argument \cite{schwarz-zaboronsky} (beyond the fact that it applies in the more general CS context). Indeed, they use limits of oscillating integrals instead of the theory of distributions. However their result can be re-interpreted (by viewing regularized limits of oscillating wavefunctions as generalized functions) as also constructing a distribution $D_{SZ}$ such that, once again,
$$Q D_{SZ} = \omega - \sum_{x\in X^Q} (\Loc_x \omega)\cdot \delta_x.$$
Interestingly, the distribution one obtains by interpreting their result in this language is different from ours. For example in the lowest-dimensional nontrivial example, which is $(2,2)$-dimensional case, our ``antiderivative'' distribution $D$ has singularities of the form $\frac{1}{x+iy}$ where $x,y$ are the even coordinates, while $D_{SZ}$ would have regularized singularities of the form $\frac{1}{x^2 + y^2}.$

\subsection{Splitting subgroups and CS volumes of homogeneous superspaces}
In the last part of the paper we consider the case when $X=G/K$ where $G$ is a quasireductive complex algebraic supergroup and $K$ is a connected 
quasireductive subsupergroup.
In this case $X$ is a smooth affine algebraic supervariety and also analytic complex supermanifold. In many cases the underlying manifold $\uX$
has a compact
real form $\uX_{\mathbb R}$ and then $X$ has a canonical structure of CS manifold, see \cite{V}.

Let $\mathfrak g$ and $\mathfrak k$ denote the Lie superalgebras of $G$ and $K$ respectively, and $p\in X$ be the point corresponding to the coset $eK$.
Then $T_pX$ can be identified with $\mathfrak p=\mathfrak g/\mathfrak k$ as a $K$-module. Let us assume that $K\subset SL(\mathfrak p)$. This happens
for example if $\mathfrak p$ is a self-dual $K$-module or if $\mathfrak k=[\mathfrak k,\mathfrak k]$. Then we can fix a $K$-invariant volume form on
$T_pX$ and using left translations extend it to a $G$-invariant volume form $\omega$ on $X$. In this case the CS integral
$I(f)=\int_Xf\omega$ defines a $G$-equivariant map $\mathbb C[X]\to\mathbb C$. If the CS volume $I(1)$ is not zero then $K$ is a splitting subgroup
of $G$.

We consider several examples of $X=G/K$ such that $\uX_{\mathbb R}$ is a (partial) flag manifold for $\uG$ and $T_pX$ is a self-dual $K$-module.
To compute the CS volume of $X$ we choose an odd
element $u\in\mathfrak k_1$ such that $[u,u]\in\operatorname{Lie}\uK$ and consider the corresponding odd vector field $Q=L_u$ on $X$. If $u$ is generic
we can show that $X^Q$ is finite and $\Loc_x\omega$ is the same at all points $x\in X^Q$. Now using our main Theorem \ref{thm:main} we obtain that
CS volume of $X$ is not zero and hence $K$ is splitting.

If $\mathfrak g$ is a finite-dimensional Kac-Moody superalgebra one can choose $K$ such that the defect subgroup $D$ is a splitting subgroup in $K$.
By transitivity of the splitting property we obtain that $D$ is splitting in $G$.

\subsection{Acknowledgments}
The authors would like to thank Alex Sherman, whose work on volumes of homogeneous superspaces with the first author made the present work possible, and for many fruitful discussions. The first author would like to thank Albert S. Schwarz for his helpful explanations of the original localisation result. The second author would like to thank Giovanni Felder, David Kazhdan, and Alexander Polishchuk for discussions that contributed to the present project. The first author was supported by NSF grant 2001191. The second author is covered by ERC grant 810573 and would like to thank the staff at IHES for the pleasant environment in which his portion of the work was completed. 

\subsection{Structure of the paper}
We first prove the localization theorem for a \emph{linear} $\GQ$-equivariant CS manifold $X$, i.e., $X = V_{CS}$ for $V$ an oriented CS linear space with $\GQ$-action. In this case, as fixed points form a linear subspace, the condition of isolated fixed points implies that $X^Q$ is supported at the point $\mathbf{0}\in V_{CS},$ and the condition of nondegeneracy is equivalent to the condition that the lie algebra element $Q\in \glie$ acts invertibly on $V$. In section \ref{sec:lin_alg}, we classify such representations of $\GQ$, which we call \emph{nondegenerate}. We also write down the linear ``localization functional'' $\Loc: \sDet(V^*)\to \cc $ (which depends only on the $\glie$-action and the orientation), and using this functional we write down a precise formula for the local contributions in the localization theorem associated to fixed points on a more general manifold. This allows us to write down a precise localization theorem, Theorem \ref{thm:main_precise}. Note that throughout this section we study (for convenience of notation) the linear $\GQ$-representation $W$ consisting of \emph{linear functions} on our linear space $X$, so that $X$ is the CS space underlying the dual vector space $W^*$ and not $W$ itself.

Then, in section \ref{sec:main_lin}, we prove the precise theorem, Theorem \ref{thm:main_precise}, for $X$ a linear $\GQ$-equivariant CS space. Since the $\GQ$-action on $X$ has a single fixed point, the localization theorem has only one local contribution and states that for any $\GQ$-equivariant form $\omega,$ we have 
$$\int_X \omega = \Loc_W(\omega\mid_\mathbf{0}),$$
where we view $\omega\mid_\mathbf{0}$ as an element of $\Ber_X\mid_{\mathbb{0}} = \sDet(T^*_X\mid_{\mathbf{0}}) = \sDet(V^*),$ and $\Loc_W:\sDet(V^*)\to \cc$ is the localization functional defined in section \ref{sec:lin_alg}.

The proof proceeds by induction on dimension, and using the theory of distributions and generalized functions on CS manifolds. We develop this theory in Section \ref{sec:distributions}. This theory is almost fully analogous to the theory in the even case, with one key difference: namely, that the delta distribution at the origin on a purely odd space $\rr^{0,n}$ (always viewed as a CS space) is actually a smooth distribution, given by the Berezinian,
$$\delta_0 = \theta_1\cdots\theta_n\{\mid d\theta_1\ldots \theta_n\}.$$ 

Via this formalism we show (by induction on dimension) that for any $\GQ$-equivariant distribution $\omega$ on $V_{CS},$ the difference of distributions $$\omega - \Loc_W(\omega\mid_\mathbf{0})\cdot \delta_{\mathbf{0}}$$ is a total derivative on $X$ of another distribution, and hence has integral zero. Since $\int_X \delta_{\mathbf{0}} = 1,$ this implies the localization theorem in this case. 

Finally, in section \ref{sec:beta} we finish the proof of the localization theorem by reducing the localization result for general CS manifold $X$ to the case of a linear $\GQ$-representation. This consists of two steps. First, in \ref{acyc_proof}, we prove the \emph{acyclicity lemma}, a result of independent interest, which states that if $X$ is a manifold with $\GQ$-action and $X^Q$ is empty (i.e., the field $Q$ is everywhere nondegenerate), then the result holds. Second, we show that in the neighborhood of any isolated nondegenerate fixed point $x\in X^Q,$ there is an isomorphism of $\GQ$-equivariant spaces between a neighborhood of $X$ and a nondegenerate CS $\GQ$-representation taking $x$ to $\mathbf{0}.$ Using these arguments, we can decompose any $X$ as a union of $\GQ$-equivariant opens which are equivariantly isomorphic to linear representations (on which the localization theorem follows from section \ref{sec:main_lin}) and opens on which $\GQ$ acts without fixed points (on which the localization theorem follows from the acyclicity lemma). We interpolate between these two using a bump function, and deduce the precise localization result, \ref{thm:main_precise}. 

In section \ref{sec:grassmannians}, we apply this theorem to suitable odd vector fields on homogeneous superspaces to deduce splitting results for super-Lie groups.


\section{Linear algebra}\label{sec:lin_alg}
Suppose $\GQ$ is a $Q$-group and $V$ is a representation of $\GQ$, always assumed finite-dimensional in this section. We say that $V$ is \emph{nondegenerate} if $Q$ acts invertibly. 

Since $Q$ is odd, any $m\mid n$-dimensional nondegenerate $\GQ$-representation $V$ must have equal even and odd dimensions, $$n=m.$$ In particular, the minimal dimension for a nonzero nondegenerate $\GQ$-representation is $(1,1).$ 

For any character $\chi$ of $T,$ we define a $(1,1)$-dimensional representation $V_\chi=\langle z,\theta\rangle$ with scalar $T$-action given by $\chi$ and $$Qz = \theta, \qquad Q\theta = \chi(Q^2)z.$$ It is clear that any $(1,1)$-dimensional $\GQ$-representation is isomorphic to some $V_\chi$. Moreover, nondegenerate $\GQ$-representations are completely reducible into a direct sum of copies of $V_\chi.$
\begin{lm}[structure theorem for nondegenerate $\GQ$-representations]
Any nondegenerate $\GQ$-representation is isomorphic to $$\bigoplus_i V_{\chi_i}$$ for some finite collection of characters $\chi_i.$
\end{lm}
\begin{proof}
Choose a $T$-eigenbasis $z_1,z_1,\dots, z_n$ of $V_0$ with eigencharacters $\chi_1,\dots, \chi_n.$ Write $\theta_i = Q(z_i).$ Since $Q$ is an isomorphism, $\theta_i$ are a basis of $V_1.$ Then $\phi_i:z\mapsto z_i, \theta\mapsto \theta_i$ gives a map $V_{\chi_i}\to V$ and $V=  \sum_i \phi_i(V_{\chi_i}).$\qedhere
\end{proof}

Recall that a \emph{CS structure} on a $\GQ$-representation $V$ is a real structure $$V_{0,\rr}\subset V_0$$ which is closed under the action of $T=\GQ^{red}.$ Note that while $Q^2\in \t{Lie}_T,$ it is an element of the complexified Lie algebra and need not preserve $V_{0,\rr}.$ Evidently, if $X$ is a $G$-manifold and $x\in X^Q,$ then $T_xX$ (as well as its dual) has a structure of a CS $\GQ$-representation. Conversely, if $V$ is a CS $\GQ$-representation, we can view $V$ as a CS $\GQ$-manifold with underlying topological space $V_{0,\rr}.$

Given a $\GQ$-representation $V$ with CS structure, recall that an \emph{orientation} on $V$ is an orientation on $V_{0,\rr}$. If $X$ is an oriented CS manifold with $\GQ$ action and $x$ is a $\GQ$-fixed point, then $T_xX$ inherits the structure of an \emph{oriented} CS $\GQ$-representation.

We define $W_\chi^{or}$ to be the oriented CS $\GQ$-representation with underlying CS space $W_\chi$ and the unique orientation such that the map $z:(W_\chi)_{0,\rr}\to \cc$ is orientation-preserving (where we view $\cc = \rr \oplus i\rr$ as a real space with standard orientation). When clear from context, we write $W_\chi$ for the oriented space $W_\chi^{or}.$

Note that the dual to an oriented CS $\GQ$-representation is again an oriented CS $\GQ$-representation.

Suppose $V$ is a nondegenerate CS $\GQ$-representation and $v\in V_0$ is a (complex) eigenvector for the $T$-action with eigenvalue $\chi.$ Then $Q^2v =\chi(Q^2)v,$ where we view $Q^2\in \tlie_\cc$ as a Lie algebra element and $\chi\in \tlie_\cc^\vee$ as a vector in the dual lattice. Since $Q$ is invertible, we must have $\chi\neq 0.$ Now $\bar{v}$ is an eigenvector with eigenvalue $\bar{\chi}=-\chi$ (since the $T$-action is compatible with real structure), so $v,\bar{v}$ are linearly independent, and thus a nontrivial CS $\GQ$-vector space must be at least $(2,2)$-dimensional. Write 
$$W_\chi : = V_\chi\oplus V_{-\chi}.$$ 
Choose a basis vector $z$ of $V_{\chi,0}\subset W$ and $\bar{z}$ of $V_{-\chi,0}\subset W.$ Then the complex conjugation $z\mapsto \bar{z}$ on $W_{\chi,0}$ turns $W_\chi$ into a CS $\GQ$-representation.
\begin{lm}\label{sdf}
\begin{enumerate}
\item\label{sdf1} Every nondegenerate CS $\GQ$-representation $W$ is a direct sum of copies of $W_\chi,$ and two such sums are isomorphic if and only if they differ only by replacing copies of $W_\chi$ by $W_{-\chi}.$
\item\label{sdf2} Every oriented nondegenerate CS $\GQ$-representation $W$ is a direct sum of copies of $W_\chi^{or},$ and two such sums are isomorphic if and only if they differ only by an even number of replacements of $W_\chi^{or}$ by $W_{-\chi}^{or}.$
\end{enumerate}
\end{lm}
\begin{proof}
Choose a partition $\Lambda_T\setminus \{0\} = \Lambda_+\sqcup -\Lambda_+$ of the character lattice of $T$ into positive and negative characters. Let $W_{0,+}$ be the eigenspace of $W$ corresponding to $\Lambda_+,$ and let $z_1,\dots, z_{n/2}$ be a basis. Then (since $\bar{\chi}=-\chi$ for any nonzero character), the elements $z_1,\dots, z_{n/2}, \bar{z}_1,\dots, \bar{z}_{n/2}$ form a basis of $W$. The subspace spanned by $z_i, \bar{z}_i, Q(z_i), Q(\bar{z}_i)$ is a subrepresentation (since $Q^2$ acts by scalars in a single eigenspace) isomorphic to $W_{\chi_i},$ and $W$ is a sum of such. Any other sum decomposition must be isomorphic as a $T$-representation, hence must have the same characters $\chi_i$ up to sign. And if we are in the oriented context, switching a $W_\chi^{or}$ to a $W_{-\chi}^{or}$ flips the orientation, hence there must be an even number of sign changes between the two sets of characters.
\end{proof}

In the remainder of the section, we will write down a certain linear-algebraic invariant of an oriented CS $G$-representation $W$ which will be responsible for the local contribution in the localization formula for $W = T^*_{x,X}$ for $x\in X$ a $Q$-fixed point. This linear contribution depends on a notion of ``Pfaffian,'' a scalar invariant of an oriented CS $\GQ$-representation, which is a square root of $Q^2\mid_{W_0}$ (up to some factors of $i$). We begin by writing down down this scalar invariant.

\begin{defi}
Suppose $W$ is an oriented nondegenerate CS $\GQ$-representation, such that $$W \cong \bigoplus_{i=1}^n W_{\chi_i}.$$ Write $$\Pf(W): = \prod_i \chi_i(Q^2).$$
\end{defi}
Note that by part \ref{sdf2} of Lemma \ref{sdf}, any oriented nondegenerate CS representation $W$ has such a decomposition and two different decompositions differ by an even number of signs, thus $\Pf(W)$ is well defined. Note also that because of nondegeneracy, $\chi(Q^2)\neq 0$ for any torus character $\chi$ appearing in $W$, so $\Pf(W) \neq 0.$

The reason we call this expression the \emph{Pfaffian} is the following.
\begin{prop}
Assume $Q^2\mid{W_{0,\rr}}$ is real, and we specify a choice of metric on $W_{0,\rr}$ compatible with the $T$-action, such that $Q^2\mid_{W_{0,\rr}}$ is skew-symmetric. Then $\Pf(W)$ is isomorphic to the ordinary Pfaffian of the matrix $Q^2\mid_{W_{0,\rr}}.$ 
\end{prop}
\begin{proof}
Since both functionals are multiplicative under direct sum and $W$ is a direct sum of copies of $W_\chi$, it suffices to check that this fact holds for $W_\chi,$ where it is obvious.
\end{proof}

Now recall that the local contribution in the integral formula is a linear functional on the fibre of the Berezinian bundle $\sDet(T^*_{x,X})$ for $X$ a fixed point. So, for $W = T^*_{x,X}$ a nondegenerate oriented $\GQ$-representation with CS structure, we must produce a class $\Loc_W\in \sDet(W)^{-1} = \det(W_0)^{-1}\otimes \det(W_1) = \hom(\det(W_0),\det(W_1)).$ The determinants of the two isomorphisms $\alpha: = Q\mid_{W_0}:W_0\to W_1$ and $\beta = (Q\mid_{W_1})^{-1}:W_0\to W_1$ give two different candidates for such a class, which differ by the factor $\det(\beta^{-1}\alpha) = \det(Q^2\mid_{W_0}).$ The local contribution $\Loc_W$ turns out to be (up to a fixed factor) a geometric mean between these two natural classes. To write down such a geometric mean, we need a root of the determinant, which is provided by the Pfaffian. We are now ready to define the localization form.
\begin{defi}\label{def:linloc} We define the localization functional of $W$ to be
\begin{align}
\Loc_W : = (2\pi)^{n/2} \Pf(W)^{-1} \det(Q\mid_{W_0})\in \det(W_0)^{-1}\det(W_1) = \sDet(W^*).
\end{align}
\end{defi}
Given an element $\omega\in\sDet(W),$ we write $\Loc_W(\omega)$ for the pairing between $\Loc_W$ and $\omega$. In the localization formula, $W = T^*_{X,x}$ and $\omega$ is the fiber of a volume form.

We record here the following obvious proposition. Note that if $$0\to W\to W'\to W''\to 0$$ is a short exact sequence of super-vector spaces, then we have canonically $\sDet(W) = \sDet(W')\otimes \sDet(W'').$ Via this identification, we recurd the following obvious proposition.
\begin{prop}\label{loc_multiplicativity}
If $0\to W\to W'\to W''\to 0$ is a short exact sequence of nondegenerate and oriented CS $\GQ$-representations, then $\Loc(W) = \Loc(W')\Loc(W'').$
\end{prop}
\begin{proof}
Note that by our complete reducibility result, we can choose a splitting $W \cong W'\oplus W''.$ Both the Pfaffian and $\det(Q\mid_{W_0})$ are then obviously multiplicative under direct sum. 
\end{proof}
\begin{rmk}
While our notion of a $Q$-group $\GQ$ has a fixed choice of $Q\in \glie,$ the localization formula only depends on a $\GQ$-invariant volume form. As such it should only depend on $Q$ up to a scalar. And indeed, note that if we replace $Q$ by $Q' = cQ$ for $c\in \cc^*$ a nonzero scalar, we get $\Pf_{Q'}(W) = (2\pi)^{n/2} \Pf(W')^{-1} \det(Q'\mid_{W_1})\in \det(W_0)\det(W_1)^{-1}\in \sDet(W).$ Now $\Pf(W')$ is a Pfaffian invariant of $(Q')^2 = c^2Q^2,$ so $\Pf(W') = c^{2n/2}\Pf(W) = c^n\Pf(W).$ We also have $\det(Q')_{\mid W_1} = c^n\det(Q)_{\mid W_1},$ and these two factors cancel.
\end{rmk}

We are now ready to give a precise statement of the localization theorem.
\begin{thm}\label{thm:main_precise}[Precise form of Theorem \ref{thm:main}]
Suppose $\GQ$ is a $Q$-group, $X$ is a CS $\G/q$-manifold that satisfies the conditions of Theorem \ref{thm:main} (i.e., such that the action has isolated and nondegenerate fixed points) and $\omega$ is a $\GQ$-equivariant form. Then $$\int\omega = \sum_{x\in X^Q} \Loc_{T_x^* X}\omega\mid_{\{x\}}.$$
\end{thm}
Here $T_x^*X$ has CS $\GQ$-representation structure induced from the action of $\GQ$ on $X$ and this structure is nondegenerate because of the nondegeneracy, hence we are in a situation where $\Loc_W$ is well-defined. 

\subsection{$(2,2)$-dimensional case, in coordinates}\label{sec:2-dim_lin}
Suppose $\GQ$ is a $Q$-group, $\chi$ is a $T$-character and $W = W_\chi$ is a two-dimensional CS $\GQ$-representation. Let $z$ be a torus $\chi$-eigenvector and $\bar{z}$ its conjugate. We can assume WLOG (up to switching $z$ and $\bar{z}$) that $z$ gives an oriented map $W_{\chi, 0}\to \cc.$ Then the $\tlie$-action is given by $t(z) = i\chi(t) z, t(\bar{z}) = -i\chi(t)$ for $t\in \tlie_\rr.$ Write $\lambda: = \chi(Q^2),$ which is nonzero by nondegeneracy but not necessarily real. Write $\theta = Qz, \bar{\theta} = Q\bar{z}.$ Note that there is no sense in which $\theta$ and $\bar{\theta}$ are conjugate since the space of odd linear functions $W_1^*$ has no distinguished real structure (rather, we use the bar notation for consistency with $z, \bar{z}$). However, if we had chosen a real structure on $W$ for which $Q\in\glie$ is real, then $\theta$ and $\bar{\theta}$ would be conjugate. 

For the reader's convenience, we record the full $\glie$-action in this basis.
\begin{align}
t(z) = i\chi(t) z,\quad t(\theta) = i\chi(t)\theta, \quad t(\bar{z}) = -i\chi(t) z, t(\bar{\theta}) = -i\chi(t) \theta\\
Q(z) = \theta \quad Q(\bar{z}) = \bar{\theta}\\
Q(\theta) = i\lambda z \quad Q(\bar{\theta}) = -i\lambda \bar z.
\end{align}
Note that in this basis, 
$Q^2 = \begin{pmatrix}i\lambda&0\\0&-i\lambda\end{pmatrix}.$ 

We compute the Pfaffian, using $W_{0,+} = \langle z\rangle$.
 The orientation assumption on $z$ then gives 
$$\Pf_W = \det\left(\frac{ Q^2\mid_{W_{0,+}}}{i}\right) = \lambda.$$
The Berezinian term in the integral formula is $\det(Q\mid_{W_0}) \in \sDet(W)^{-1} = \Lambda^2W_1\Lambda^2W_0^*.$ Recall that the determinant of an isomorphism $M:V\to W$ of even spaces is given by the wedge product of a basis $x_i$ of $V$ divided by the wedge product of $M(x_i),$ giving $\det(Q\mid_{W_0}) = \frac{\theta\wedge \bar{\theta}}{z\wedge \bar{z}}.$ Finally, the integral form is
$$\Loc_W = 2\pi \frac{\theta\wedge \bar{\theta}}{\lambda z\wedge \bar{z}}.$$

\begin{rmk}[Real basis in real case]
Note that we can always choose a real basis of $W_0$ given by $x = \frac{z+\bar{z}}{2}, y = \frac{z-\bar{z}}{2i},$ though it is less convenient for calculations. 
If all of $W$ has a distinguished real structure compatible with $\GQ$-action, and $Q$ is real, then we can also choose the real basis on $W_1$ given by $\theta_x : = \frac{\theta + \bar{\theta}}{2}, \theta_y: = \frac{\theta - \bar{\theta}}{2i}.$ In this basis, we have
$$Q^2 = \begin{pmatrix}0&\lambda\\-\lambda &0\end{pmatrix}.$$
The localization form in this basis is 
$$\Loc_W = 2\pi \frac{\theta_x\wedge \theta_y}{\lambda x\wedge y}.$$
\end{rmk}

Note that if $W=V\oplus V^*$ then $W$ has a canonical volume form $\omega$ which pairs $\sDet V$ and $\sDet V^*$. If $u_1,\dots,u_n$ is basis
  in $V_0$ and $v_1,\dots,v_n$ is a basis in $V_1$, and $u_1^\vee,\dots,u_n^\vee, v^\vee_1,\dots,v^\vee_n$ is the dual basis in $V^*$
  then $$\omega=\frac{u_1\wedge\dots\wedge u_n\wedge u_1^\vee\wedge\dots\wedge u_n^\vee}{v_1\wedge\dots\wedge v_n\wedge v^\vee_1\wedge\dots\wedge v^\vee_n}.$$
  The volume form $\omega$ does not depend on the choice of a basis.
  \begin{lm}\label{canonical_form}
    Let $W=V\oplus V^*$ and $\dim V=(n|n)$.  Assume that one can choose an $u^2$ eigenbasis $z_1,\dots, z_n$ in $V_0$ such that $\bar z_i=z_i^\vee$.
    Then
    $\Loc_W\omega=\left(\frac{2\pi}{\mathbf i}\right)^n$. 
 \end{lm}
 \begin{proof}
   Let $\mathbf i\lambda_1,\dots,\mathbf i\lambda_n$ be the corresponding eigenvalues.
   Set $\theta_i=Qz_i$. Then
   $$\langle Q\bar z_i, \theta_j\rangle=-\langle \bar z_i,Q\theta_j\rangle=\langle z_i, \mathbf i\lambda_jz_j\rangle=\mathbf i\lambda_i\delta_{i,j}.$$
   Therefore $\bar\theta_i=\mathbf i\lambda_i\theta_i^\vee$. By Proposition \ref{loc_multiplicativity} we have
   $$\Loc_W=\frac{(2\pi)^n}{\prod \lambda_j}\frac{\theta_1\wedge\dots\theta_n\wedge\bar\theta_1\wedge\dots\wedge\bar\theta_n}{z_1\wedge\dots z_n\wedge\bar z_1\wedge\dots\wedge\bar z_n}=\left(\frac{2\pi}{\mathbf i}\right)^n\frac{\theta_1\wedge\dots\theta_n\wedge\theta_1^\vee\wedge\dots\wedge\theta_n^\vee}{z_1\wedge\dots z_n\wedge z^\vee_1\wedge\dots\wedge z^\vee_n}.$$
   This implies the statement.
 \end{proof}
 \begin{rmk}\label{distinct} The assumption of Lemma \ref{canonical_form} holds if $V=\bigoplus V_{\chi_j}$ if $\dim V_{\chi_j}=(1|1)$ for every $j$
   and $\chi_i\neq\chi_j$ for $i\neq j$.
   \end{rmk}

\section{Calculation on linear space.}
In this section we prove the main theorem for a linear local model.
\begin{thm}\label{thm:linear_calc}
Suppose $\GQ$ is a $Q$-group and $W$ is an oriented $(n, n)$-dimensional nondegenerate CS $\GQ$-representation. Then Theorem \ref{thm:main_precise} holds for the linear manifold $X = W^*.$
\end{thm}

In the remainder of this section we prove this local theorem. For $v\in W_0,$ we write $\mathbf{v}$ for $v$ considered as a point of the linear space $X$ (with underlying manifold $\uX = W_{0,\rr}$). In this case, there is a unique $Q$-fixed point $\mathbf{0},$ which obviously satisfies the conditions of Theorem \ref{thm:main}. We prove the theorem by induction on $n$. Recall that the CS and nondegeneracy conditions imply that $n$ must be even. So suppose Theorem \ref{thm:main_precise} holds for $(n-2,n-2)$-dimensional manifolds. We argue using the theory of distributions. 

\subsection{Distributions and generalized functions}\label{sec:distributions}
Recall that a function on a super-ringed space $X = (\uX, C^\infty_X)$ is \emph{compactly supported} in $\uU\subset \uX$ if the restriction of $f$ is zero on the complement of a compact subset of $\uU$, and write $C^\infty_c(\uU)$ for the space of compactly supported functions. 
\begin{defi} A \emph{distribution} $\xi$ on $U$ is a functional on compactly supported functions $\xi:f\mapsto \langle \xi,f\rangle$ i.e., an element of $C^\infty_c(\uU)^\vee.$ 
\end{defi}
When $X$ admits partitions of unity, distributions on $X$ form a sheaf; it is $\zz/2$-graded and pre-composition with multiplication by functions makes it into a $C^\infty_X$-module. We will also need a dual notion, of \emph{generalized functions}.
\begin{defi} For $X$ a supermanifold, we define the sheaf of \emph{generalized functions} $$C^{-\infty}_X: = \t{Distr}_X\otimes_{C^\infty_X} \Ber_X^{-1}$$ (equivalently, this is the dual sheaf to compactly supported Berezinian sections). 
\end{defi}
Any Berezinian can be viewed as a distribution ($\omega$ goes to the functional $f\mapsto \int f\omega$), giving a map of $C^\infty$-module sheaves $\Ber_X\to \t{Distr}_X,$ thus also $C^\infty_X\to C^{-\infty}_X.$ This is the super-analogue of the inclusion of smooth functions into generalized functions (indeed, in the super case it is also true that $C^{-\infty}$ is a completion of $C^{\infty}$ in a suitable topology). Note that for the purely odd linear CS vector space $Y = \Pi W_1,$ all functions are compactly supported so we have $\t{Distr}_Y = C^\infty(Y).$ The integral pairing $\Ber_Y\otimes \oo_Y\to \cc$ is a perfect pairing of finite-dimensional vector spaces, so we have canonical isomorphisms $\t{Distr}_Y = \Ber_Y, C^{-\infty}_Y = C^\infty_Y.$ Applying this argument fiberwise on the family $X \to X^{red}$ (with fibers $W_1$), we see that we have canonical isomorphisms of sheaves of vector spaces over $\uX = V_{0,\rr},$ and $C^{-\infty}_W\cong C^{-\infty}_{W_0}\otimes C^\infty \Pi V_1.$

Note that since differentiation preserves the class of compactly supported functions, any derivation $Q$ on $X$ induces a derivation on $\t{Distr}$ by $\langle Q\xi,f\rangle: = (-1)^{|Q||f|}\langle\xi,Qf\rangle.$ This is a right action of the Lie algebra, $TX$, or equivalently, an action of $TX^{op}$. We get a (left) action of $TX$ on $C^{-\infty}$ in a similar way.

Note that if a distribution $\xi$ is compactly supported, we can choose a compactly supported bump function $\tau$ equal to $1$ on the support of $\xi,$ and define $$\langle \xi, f\rangle: = \langle \xi, \tau f\rangle$$ for \emph{any}, not necessarily compactly supported function $f$. This is obviously independent of the bump function $\tau,$ and we obtain a canonical pairing between compactly supported distributions and \emph{all} functions. In particular we can define the \emph{integral} of a compactly supported distribution $\xi$ on $X$ by $$\int_X \xi : = \langle \xi,1\rangle.$$ 
We then have the familiar identity that a compaclty supported distribution which is a ``total derivative'' of another distribution has zero integral. In other words, if $\omega$ is a compactly supported distribution and $Q$ is a vector field (even or odd), then we have $$\int_X Q(\xi) = \pm\langle \xi, Q(1)\rangle = 0.$$

We will use the fact that (as in even geometry), generalized functions pull back along smooth maps. Namely, if $\pi:X\to Y$ is a smooth submersion of CS supermanifolds and $f$ is a generalized function on $Y,$ then we define a generalized function $\pi^*f$ on $X$ as follows. Recall that for any compactly supported volume form $\omega$ on $X$, we can integrate it along the fibers of $\pi$ to produce a new volume form $\pi_!\omega\in \Gamma(\Ber_Y).$ We then define $\pi^*f$ to be the form characterized by $\int_X \pi^*f\cdot \omega : = \int_Y f\cdot \pi_!\omega.$ 

The pullback map is compatible with differentiation, in the sense that if $Q$ is a differential on $Y$ and $\tilde{Q}$ is a lift of $Q$ to $X$ then $\tilde{Q}\pi^*f = \pi^*(Qf)$. 

\subsection{$\delta$ distributions and $\delta$ functions}
Let $V$ be any CS vector space and let $X = V_{CS}$ be the corresponding linear CS manifold. We have a canonically defined delta distribution $\delta_{\mathbf{0}}$ defined by the formula $$\langle \delta_{\mathbf{0}}, f\rangle = f(0).$$

If $\omega\in \sDet(V^*)$ is a constant Berezinian form, we have a generalized delta function $\delta_{\mathbf{0}}/\omega$ (unlike the delta distribution, the delta function is only canonical up to scalar). If $V = \rr^{0,n}$ is purely odd, then $C^{-\infty}(V_{CS}) = C^\infty(V_{CS}),$ so the delta function is an ordinary function. Choosing coordinates functions $\theta_i,$ we have the following important (and obvious) identity:
\begin{align}\label{distrpoint}\frac{\delta_{\mathbf{0}}}{\{\mid d\theta_1\cdots d\theta_n\}} = \theta_1\cdots \theta_n.\end{align}
In other words, when working with delta functions we can ``move factors of $\{:\theta_i\}$ in the denominator to factors of $\theta_i$ in the numerator''. More generally, suppose $V = V'\oplus V''$ is a direct sum decomposition. Write $X = V_{CS}^*, B = {V'}_{CS}^*, Y = {V''}_{CS}^*.$ Let $i:Y\to X$ be the inclusion and $\pi:X\to F$ be the quotient (so $Y\to X\to B$ is a fibration sequence). Note that elements of $\sDet(V), \sDet(V'), \sDet(V'')$ can be viewed as constant Berezinians on the corresponding spaces. We have a canonical isomorphism $\sDet(V) = \sDet(V')\otimes \sDet(V'')$. Now, given any Berezinian section $\omega''\in \Ber(Y),$ we can define a distribution $$\delta_{Y}\cdot \omega''\in \t{Distr}_X$$ with $\langle \delta_{Y}\cdot \omega'',f\rangle: = \int_Y (f\mid_Y)\cdot \omega''.$ We can view the symbol
$$\delta_{Y}\in \Gamma(Y,\sDet(N_{Y\subset X}^*))$$
formally as a ``relative Berezinian'' on $Y$, i.e., as a section of the determinant of the conormal bundle $N_{Y\subset X}^*$ of $Y$ in $X.$ Via this point of view, given an element $\kappa\in \sDet(V')^{-1},$ we can define the generalized function $\kappa\cdot \delta_{Y\subset X}\in C^{-\infty}(X)$ by $\langle \kappa, \omega\rangle : = \int_Y \omega\otimes\kappa$ for any compactly supported Berezinian section $\omega \in \Gamma_{cs}(X, \Ber)$ where we view $\omega\otimes \kappa\in \Gamma_{cs}(\uY,\Ber(Y))$ via the natural isomorphism 
$$\Ber(X)\mid_Y\otimes \sDet(V')^{-1}\cong C^\infty(X) \otimes \sDet(V)\otimes \sDet(V')^{-1}\cong \Ber(Y).$$

\subsection{The interpolating distribution $\sigma$ and the $(2,2)$-dimensional case}\label{sec:main_lin}

We now specialize to the $(2,2)$-dimensional case. Fix a character $\chi$ of the torus $T = \GQ^{red}$ and take $W = W_\chi,$ a $(2,2)$-dimensional $\GQ$-module. Write $X = W^*_{CS},$ the linear CS manifold associated to $W^*.$ From Section \ref{sec:2-dim_lin}, we have a basis of linear functions $z,\bar{z}, \theta, \bar{\theta}$ on $X = W^*,$ and a scalar $\lambda\in \cc$ such that $Q^2z = i\lambda z.$
Write $\partial, \bar{\partial}$ for the basis of even derivations on $X_\chi$ dual to $dz, d\bar{z}$ (i.e., determined by the formula 
\begin{align*}
\partial(z) = 1,\quad \bar{\partial}(z) = 0\\
\partial(z) = 0, \quad \bar{\partial}(\bar{z}) = 1, )
\end{align*}
and similarly for the basis of dual derivations $\partial_\theta, \bar{\partial}_\theta.$
Then (using the coordinate calculations in Section \ref{sec:2-dim_lin}) we obtain the formula
\begin{align}\label{formforQ} Q = \theta \partial + \bar{\theta}\bar{\partial} + i\lambda(z \partial_\theta - \bar{z}\bar{\partial}_\theta).
\end{align}

We identify its underlying space with $\cc\cong \rr^2,$ with real coordinates $x = \frac{z+\bar{z}}{2}, y = \frac{z-\bar{z}}{2i}$. The functions $z,\bar{z}$ can are linear complex-valued functions on $\cc$; by a standard abuse of notation, we use the same letters to denote their pullbacks to $X_\chi$ via the linear projection $X_\chi\to \cc = X_\chi^{red}$. It is a standard result in analysis that we can view $$\frac{1}{z}$$ as a generalized function on $\cc$, defined as the $\epsilon\to 0$ limit of the continuous functions $\frac{1}{z} (1-\chi_{D_\epsilon})$ (for $\chi_{D_\epsilon}$ the characteristic function of a disk). We have the following standard formula from harmonic analysis, working over the even manifold $\cc = X_\chi^{red}$: 
\begin{align}\label{polediff}
\bar{\partial}\frac{1}{z} = -2\pi i\frac{\delta_0}{dzd\bar{z}}.
\end{align}
(Equivalent to $\bar{\partial}\frac{1}{z} = \pi\frac{\delta_0}{dxdy}.$) Here $\bar{\partial}$ is the complex derivation characterized by $\bar{\partial}z = 0, \bar{\partial}\bar{z} = 1.$ The projection $X_\sigma\to \rr^2$ lift $1/z$ to a function on $X_\sigma$ (and hence also on $X$). Now, working over the supermanifold $X_\chi,$ write
\begin{align}\label{sigmaeq}
\sigma(z): = \frac{\theta}{i\lambda z},
\end{align}
a function on $X_\chi$ (which also lifts to $X$). Then via a computation we obtain the following crucial equation.
\begin{lm}\label{lm:interp_eq}
We have $$Q\sigma = 1 - \frac{2\pi}{\lambda}\frac{\delta_{\mathbf{0}}}{\{dz d\bar{z}\mid d\theta d\bar{\theta}\}}.$$
\end{lm}
\begin{proof}
We simply combine formulas (\ref{formforQ}) and (\ref{polediff}) to perform the following calculation,
\begin{align} 
Q\frac{\theta}{z} = \left(\theta \partial + \bar{\theta}\bar{\partial} + i\lambda(z \partial_\theta - \bar{z}\bar{\partial}_\theta)\right)\frac{\theta}{z} &= \\
\theta^2(\cdots) -2\pi i \theta\bar{\theta} \frac{\delta^{red}}{dzd\bar{z}} + i\lambda(1 - 0) = i \lambda - 2\pi i \frac{\delta_{z = 0}}{dzd\bar{z}},
\end{align}
which we then divide by $i\lambda.$

\end{proof}

Lemma \ref{lm:interp_eq} implies that
$$\lambda - 2\pi \frac{\delta_0}{\lambda dzd\bar{z}}$$
is a total $Q$-derivative of a function, and so if $\omega$ is a compactly supported $Q$-equivariant Berezinian (either on $X_\sigma$ or on $X$ itself), then
$$(\lambda-2\pi \frac{\delta_{z = 0}}{\lambda dzd\bar{z}})\omega = Q(\sigma\omega)$$
has zero integral. 

In particular, we deduce that $$\lambda\int \omega = 2\pi \int \frac{\omega}{dzd\bar{z}}
$$
This implies the localization theorem for the $(2,2)$-dimensional space $X_\chi$. But since generalized functions lift along smooth maps, we can lift the equation in Lemma \ref{lm:interp_eq} from $\rr^{2,2}$ to all of $W_{CS}$ to obtain the induction step in the proof of theorem \ref{thm:linear_calc}.

\subsection{The induction step}
Let $W = \prod_{i=0}^{n/2-1} W_i$ be a decomposition of $W$ as a product of $(2,2)$-dimensional CS representations, with coordinates $z_i, \bar{z}_i, \theta_i, \bar{\theta}_i$ as in \ref{sec:2-dim_lin} Let $W_{>0} = \prod_{i\ge 1} W_i\subset W$ be the locus where $z_0=\bar{z}_0=\theta_0=\bar{\theta}_0=0.$ We view $W_0$ with coordinates $z_0,\bar{z}_0, \theta_0, \bar{\theta}_0$ as a quotient of $W$ (with fiber $W_{>0}$ at $0$). We can thus view any generalized functions $f$ on $W_0$ as a generalized function on $W$ (that only depends on the first $(2,2)$ coordinates). When there is no chance of confusion, we will denote a function and its pullback by the same letter. Write $X = W^*_{CS}$ for the affine space we are working on and $Y = W_{>0 CS}^*$ for the codimension $(2,2)$ subspace that maps to $0\in \rr^{2,2}$ under projection to the $0$th coordinate.
From pulling back the formula in Lemma \ref{lm:interp_eq}, we have
\begin{align}
\int_X \omega &= \int_X \left(\omega - Q (\sigma \omega)\right) \\ &= \frac{2\pi}{\lambda}\int_X \frac{\theta\bar{\theta}\pi^*\delta_{\mathbf{0}, W_0}}{dzd\bar{z}} \omega\\
&= \frac{2\pi}{\lambda}\int_Y \frac{\omega}{dzd\bar{z}\mid d\theta d\bar{\theta}} = \frac{2\pi}{\lambda}\Loc_{W_{>0}}(\frac{\omega\mid_{\mathbf{0}}}{dzd\bar{z}}\mid_0) \label{line3} \\
&=\Loc_{W_0}\{dzd\bar{z}\mid d\theta d\bar{\theta}\}\Loc_{W_{>0}}\frac{\omega\mid_{\mathbf{0}}}{dzd\bar{z}{d\theta d\bar{\theta}}} = \Loc_W\omega\mid_{\mathbf{0}}\label{line4},
\end{align}
where equation \ref{line3} follows from the induction hypothesis and equation \ref{line4} follows from the multiplicativity of Loc, Proposition \ref{loc_multiplicativity}. This concludes the proof of Theorem \ref{thm:linear_calc}. \qed


\section{Reduction to linear case}\label{sec:beta}
In this section we reduce the general case of Theorem \ref{thm:main_precise} to the local calculation of the previous section. The main component of the proof is the acyclicity lemma, Lemma \ref{lm:acyc}, which will imply that the localization theorem holds for Berezinians supported outside of $X^Q.$ 

\subsection{Proof of the acyclicity lemma} \label{acyc_proof}
We give a reminder of the statement of the lemma.
\begin{clemma}[$Q$-acyclicity, Lemma \ref{lm:acyc}]
Suppose $X$ is a $\GQ$-equivariant CS manifold and $Q$ is everywhere nonvanishing. Then there exists an odd function $\beta \in C^\infty(X)_1$ such that $Q\beta = 1.$
\end{clemma}
\begin{proof} 
Our proof closely follows the proof in \cite{schwarz-zaboronsky} (with the exception that we replace an argument involving a choice of metric by Lemma \ref{lm:dred} below ). We first construct an odd function $\alpha$ on $X$ which satisfies the condition up to nilpotents, $(Q\alpha)^{red} = 1.$ 

Let $T^{red}_1$ be the odd component of $T^{red},$ i.e., the normal bundle to $X^{red}$ in $X$. 

Let $\Omega^{red}$ be the restriction of the cotangent bundle $\Omega$ to $X^{red}$ and let $\Omega^{red}_1$ be its odd part, which is canonically isomorphic to the conormal bundle $N_{X/X^{red}}^*$. Let $\rho:\Gamma(\uX,\Omega)_1\to \Gamma(\uX, \Omega^{red}_1)$ be the restriction map, and let $$d^{red}_1:= \rho\circ d:\Gamma(\uX, C^\infty_X)\to \Gamma(\uX, \Omega^{red}_1)$$ be the composition, that takes an odd function to its differential in the normal directions.
\begin{lm}\label{lm:dred}
The restriction map $d^{red}_1:\Gamma(\uX, C^\infty_X)_1\to \Gamma(\uX, \Omega^{red}_1)$ is surjective.
\end{lm}
\begin{proof}
This result is obvious for $X = \rr^{m\mid n},$ hence true locally. Let $\{\uU_i\}_{i\in I}$ be a \v{C}ech cover of $\uX$ such that the lemma is true on each $\uU_i$ (any \v{C}ech refinement of a cover by $\rr^{m\mid n}$ will do). Let $\{\tau_i\}_{i\in I}$ be a partition of unity with $\tau_i$ supported on $U_i.$ For $\phi\in \Gamma(\uX, \Omega^{red}_1),$ let $f_i\in \Gamma(\uU_i, C^\infty_X)$ be odd functions with $d^{red}_1f_i = \phi\mid_{U_i}.$ Then $$f:=\sum \tau_if_i$$ is a locally finite sum of compactly supported functions, hence defines a global odd function on $X$. Since $d^{red}_1$ commutes with multiplication by even functions, we have $d^{red} f = \phi.$
\end{proof}
\begin{rmk}
This lemma implies that any CS manifold is (non-canonically) isomorphic to the total space of the bundle $T^{red}_1$ on $X^{red},$ a well-known result in the real case.
\end{rmk}
Now convolving with $Q^{red}\in T^{red}_1$ gives a linear map of vector bundles $*_Q^{red}:\Omega^{red}_1\to C^\infty_{X^{red}}$ on the even manifold $X^{red},$ which is surjective since $Q$ was assumed nowhere vanishing. Surjective maps of $C^\infty$ vector bundles split, so we can choose a preimage $\phi$ with $*_Q^{red}(\phi) = 1.$ Using Lemma \ref{lm:dred}, we choose a lift $\alpha\in C^\infty(X)_1$ with $d^{red}\alpha = \phi.$ By construction, we then have $$(Q\alpha)^{red} = Q^{red}*d^{red}\alpha = 1.$$

Now write $[\alpha]$ for the $T$-average of $\alpha.$ Then $[\alpha]$ is a $T$-equivariant function, and since $Q$ is $T$-equivariant we have $Q[\alpha]^{red} = [1] = 1.$ In particular, we see that $Q[\alpha]$ is invertible. Write 
$$\beta:=\frac{[\alpha]}{Q[\alpha]}.$$
We compute: $Q\beta = \frac{Q[\alpha]}{Q[\alpha]} + \frac{Q^2[\alpha]}{\ldots} = 1,$ since $Q^2[\alpha] = 0$ by $T$-equivariance of the averaged function $[\alpha].$ 
  \end{proof}

\subsection{Reduction to a neigbhorhood}
We write down the following corollary, which implies that the localization theorem reduces to a neighborhood of the fixed points. 
\begin{cor}\label{cor:reduction_to_loc}
Suppose $X$ is a $\GQ$-manifold with $X^Q$ supported at the dimension-zero submanifold $\uY\subset \uX,$ and suppose that the localization theorem, Theorem \ref{thm:main} holds for $U = \uX, C^\infty_X\mid_{\uU}$ for $\uU$ some open $T$-equivariant neighborhood $\uU\supset \uY$ of $\uY$ in $\uX$ (note that since $\GQ$ is a nilpotent thickening of $T$ and $U$ is open, $\GQ$ automatically acts on $U$). Then the localization theorem holds for $X$. 
\end{cor}
\begin{proof}
Let $\tau$ be a ``bump function,'' such that $\tau$ has support inside $\uU$ and is equal to $1$ in a smaller neighborhood $\uV\supset \uU$ of $\uY$ (it is standard that such exist). By taking the $T$-average, we can assume WLOG that $\tau$ is $T$-equivariant (hence also $Q^2$-equivariant). Let $\uV'$ be a $T$-equivariant neighborhood such that $\uV\cup \uV' = \uX$ but $\uV'$ does not intersect $\uY.$ By the $Q$-acyclicity theorem, there is a function $\beta_{V'}$ on $V'$ with $Q\beta_{V'} = 1$ on $V'.$ Write $\beta' := \beta_{V'}\cdot \tau'.$ This function is zero on $\uV\cap \uV',$ hence can be extended (by zero) to a function on all of $X = \uV\cup \uV',$ which we also call $\beta'.$ Now outside $\uU,$ we have $\tau = 0$ so $\tau' = 1,$ and thus $\beta'\mid_{\uX\setminus \uU} = \beta\mid_{\uX\setminus \uU}$. Now suppose $\omega$ is a $\GQ$-equivariant Berezinian. Write $\omega_{\uU} : = \omega - Q(\omega\beta').$ Since they differ by a total derivative, $$\int \omega_{\uU} = \int \omega.$$ Now in the complement of $\uU,$ we have $Q\beta' = 1,$ so $$\omega_{\uU}\mid_{\uX\setminus \uU} = 0.$$ On the other hand, inside $\uV,$ we have $\tau' = 0,$ so $\beta' = 0$ and $$\omega_{\uU}\mid_{\uV} = \omega\mid_{\uV}.$$ Since $\uV$ contains the fixed point locus $\uY$, this implies that the local contributions $$\sum_{x\in X^Q} \Loc_{T_x^* X}\omega\mid_{\{x\}} = \sum_{x\in X^Q} \Loc_{T_x^* X}\omega_{\uU}\mid_{\{x\}}$$ of $\omega$ and $\omega_{\uU}$ are equal. Now by construction, $\omega_{\uU}$ is supported on $\uU$ and $\GQ$-equivariant (this because $\omega-\omega_{\uU} = Q(\omega\beta')$ is the $Q$-derivative of something $T$-equivariant, thus is both $T$-equivariant and $Q$-equivariant). Thus if the localization theorem holds for the manifold $U$, we must have $$\sum_{x\in X^Q} \Loc_{T_x^* X}\omega\mid_{\{x\}} = \int\omega_{\uU},$$ and we deduce the same result holds for $\omega.$
\end{proof}

\subsection{Local linearity} 
With Corollary \ref{cor:reduction_to_loc} and Theorem \ref{thm:linear_calc} in hand, it remains only to prove that every $Q$-fixed point of $X$ (the $\GQ$-manifold in Theorem \ref{thm:main}) has a neighborhood isomorphic as a CS $\GQ$-manifold to a neighborhood of $\mathbf{0}$ in a linear CS $\GQ$-representation. 

We give a brief reminder of the implicit function theorem for CS manifolds. Suppose $X$ is a CS manifold and $\mathbf{x}$ is a point. Let $W$ be a CS super-vector space with $W^*_{CS}$ the CS space associated to the dual vector space, with underlying manifold $W^*_{0,\rr}$. Let $f:X\to W^*_{CS}$ be a map of CS vector spaces such that $\mathbf{x}\mapsto 0$. Then pullbacks of linear functions on $W^*_{CS}$ give a map $$L = f^*\mid_{W\subset C^*(W^*)}:W\to C^\infty X.$$ Note that $L$ encodes all the information about the map. In particular, the map on underlying topological spaces is given by $x\mapsto L_x$ for $L_x\in W^*$ the functional $L_x:w\mapsto L(w)(x).$ Since odd functions vanish at points, $L_x\in W_0^*$ is even and furthermore $L_x\in W_{0,\rr},$ as $L_x$ must correspond to a topological point of the underlying manifold $\uW^* = \uW_{0,\rr}^*.$ We show in the following proposition that, conversely, a linear function $W\to C^\infty X$ that takes $W_{0,\rr}$ to functions with real values at all points uniquely determines a map $X\to W^*_{CS},$ and that moreover, this correspondence satisfies an implicit function theorem.

\begin{prop}\label{xcvb}
Suppose $L:W\to C^\infty(X)$ is a linear map such that $L(w)(\mathbf{x}) = 0$ vanishes on the marked point for all $w$ and such that $\pi_{red}\circ L\mid_{W_0}:W_0\to C^\infty(X)^{red}$ is compatible with real structure (for $\pi_{red}$ the projection map to the reduced part). 
\begin{enumerate}
\item\label{xcvb1} There is a unique map of CS manifolds $f:X\to W^*_{CS}$ with $f^*(w) = L(w)$ for any $w\in W$ viewed as a linear function on $W^*_{CS}$. 
\item\label{xcvb2} Let $d_x\circ L: W\to \Omega^1_{X,\mathbf{x}}$ be the composition of $W$ with the differential at $\mathbf{x}.$ Then $f$ is a local isomorphism near $\mathbf{x}$ if and only if $d_x\circ L$ is an isomorphism.
\end{enumerate}
\end{prop}
\begin{proof}
Part \ref{xcvb1} follows immediately from the interpretation of CS manifolds in terms of charts. Alternatively, let $L_{red,\rr}:W_{0}\to C^\infty(X)_{red}$ be the map on even, reduced functions. This is a map in the world of (real) differential geometry, and since it is compatible with real structures, it corresponds to a unique map of real manifolds $f^{red}:X^{red}\to W^*_{0,\rr}$ (since in the world of real manifolds, $C^\infty_\rr$ classifies smooth maps to $\rr^1$). Now the proposition follows from the product identity $W^*_{CS} = W^*_{0,CS}\times \Pi(W^*_1)_{CS}$ in the category of CS manifolds. 

For part \ref{xcvb2}, let $f:X\to W^*_{CS}$ be the map of CS manifolds corresponding to the linear map $L:W\to C^\infty(X)$. Let $df: T^1_{\mathbf{x}}X\to T^1_{\mathbf{0}}W^*_{CS}$ be the induced map on tangent spaces. The map $df$ is evidently adjoint to $d_x\circ L,$ and so one is an isomorphism if and only if the other is. In particular if $f$ is a local isomorphism at $\mathbf{x}$, then the induced map on tangent spaces $df\mid_{\mathbf{x}}$ must be an isomorphism as well. It remains to show the converse. Suppose $df$ is an isomorphism at $\mathbf{x}$. Then it must be an isomorphism in a neigborhood of $\mathbf{x},$ so since we are working locally, we can assume WLOG that it is an isomorhpism on all $X$. Then ordinary differential geometry implies that $f^{red}:X^{red}\to W^*_{0,\rr}$ is a local isomorphism. The induced map on the normal bundles of $X^{red}\subset X$ is now also an isomorphism. If we filter both the sheaves $C^\infty(X)$ and $C^\infty(W^*_{CS})$ by powers of the nilpotent ideal, the map of ringed spaces is compatible with these filtrations and induces an isomorphism of associated graded components (which are powers of the induced map on normal bundles). Thus $f$ is a local isomorphism. 
\end{proof}

Now suppose that $\mathbf{x}\in \uX$ is an isolated fixed point of a $\GQ$-action on $X$ with nondegenerate action. 
\begin{lm}\label{lm:local_linearity}
There is a $\GQ$-equivariant map of CS manifolds $f:X\to (T_{\mathbf{x}, X})_{CS}$ sending $\mathbf{x}$ to $0,$ which is a local isomorphism.
\end{lm}
\begin{proof}
Let $$W := T^*_{\mathbf{x}, X}$$ with induced CS structure, and let $I_x\subset C^\infty_X$ be the ideal of functions on $X$ that vanish on $\mathbf{x}.$ We have a map $D:I_x\to W$ (differentiating at $x$) which is $\GQ$-equivariant and surjective. Let $D_0:I_{x,0}\to W_0$ be its even component, which is surjective and equivariant with respect to the even part $T\subset \GQ$. Furthermore, $D_0$ factors through the reduced restriction $I_{x,0}\to I_{x,0}^{red} (= I_{x,0}\cap C^{\infty,red}),$ and the resulting map $D_0^{red}: I_{x,0}^{red}\to W_0$ is compatible with real structure. 

Choose an arbitrary \emph{real} splitting of the surjection $D_0^{red}: I_{x,0}^{red}\to W_0,$ then average with respect to the compact group $T$ to get a $T$-equivariant splitting $$L_0^{red}: W_0\to I_{x,0}^{red}$$ such that $D_0\circ L_0 = \t{Id}_{W_0}$ and $L_0^{red}$ is compatible with real structure. Repeating this argument once more for the surjection $I_{x,0}\to I_{x,0}^{red}$ (now without real structure), we get a further splitting $$L_0: W_0\to I_{x,0}^{red}$$ such that $\pi_{red}\circ L_0 = L_0^{red}: W\to I^{red}$ (and therefore, also, $L_0\circ D_0 = \t{Id}_{W_0}$.) 

Now by nondegeneracy, the action of $Q$ defines an isomorphism $Q:W_0\to W_1.$ Write $L_1: = Q\circ L_0\circ Q^{-1}: L_1\to I_{x,1}.$ Let $L = L_0\oplus L_1:W\to I_x$ be the linear map with even component $L_0$ and odd component $L_1.$ Both components are $T$-equivariant, and by construction, $L$ commutes with $Q$. Thus $L$ is a $\GQ$-equivariant splitting of $D:I_x\to W.$ Finally, because (by construction), the map $L_0^{red} = \pi_{red}\circ L_0$ is real, we see that $L$ satisfies the conditions of Proposition \ref{xcvb}, and thus corresponds to a function $f: X\to W^*_{CS}$ with pullback of linear functions satisfying $f^*(w) = L(w);$ since $L$ is $\GQ$-equivariant, $f$ is as well. Now since $L_0$ splits the surjection $D:I_x\to T^*_x,$ we see that $d_{\mathbf{x}}f$ is an isomorphism, so by the implicit function theorem $f$ is a local isomorphism near $\mathbf{x}.$ 
\end{proof}

\begin{cor}
The localization theorem holds for some open neighborhood of $X^Q$ in $X.$
\end{cor}
\begin{proof}

Now we are ready to complete the proof of the localization theorem, Theorem \ref{thm:main} (and more specifically the precise statement \ref{thm:main_precise}). From Lemma \ref{lm:local_linearity}, we see that for each $\mathbf{x}\in X^Q,$ there exists a neighborhood $\uU_{\mathbf{x}}\subset \uX$ and a bijection between $U_{\mathbf{x}}: = (\uU_{\mathbf{x}}, C^\infty_X)$ and a linear $\GQ$-space. Since $\uX^Q\subset X$ is a closed discrete subset, we can choose $\uU_{\mathbf{x}}$ to be disjoint. Write $$U = \bigsqcup_{x\in \uX^Q} U_\mathbf{x}.$$ Then from Theorem \ref{thm:linear_calc}, we see that the localization theorem holds on $U.$ We deduce using Corollary \ref{cor:reduction_to_loc} that the localization theorem holds for all of $X.$
\end{proof}

  \section{CS volumes of some homogeneous superspaces}\label{sec:grassmannians}
  
  Let $G$ be a connected complex quasireductive algebraic supergroup, i.e., $\uG$ is a reductive algebraic group. Let $K$ be a connected quasireductive
  sub supergroup. Then $X=G/K$ is an affine algebraic
  supermanifold over $\mathbb C$, \cite{MaT}. By $p\in X$ we denote the base point corresponding to the coset $eK$.  Let $\mathfrak g$ and
  $\mathfrak k$ denote the Lie superalgebras of $G$ and $K$ respectively. The tangent space
  $T_pX$ is isomorphic to $\mathfrak p:=\mathfrak g/\mathfrak k$ as a $\mathfrak k$ module, via the adjoint action of $\mathfrak k$.

  Recall that $\uG$ has a unique up to isomorphism compact form $\uG_{\mathbb R}$. Let $\uK_{\mathbb R}:=\uK\cap\uG_{\mathbb R}$. Assume that
  $\uX_{\mathbb R}:=\uG_{\mathbb R}/\uK_{\mathbb R}$ is a totally real submanifold of $\uX$,
  in the sense of \cite{V} and hence by Theorem 3 in \cite{V} $X$ has a canonical structure of CS manifold. If we assume further that $K$ preserves a
  volume form $\omega\in\sDet(\mathfrak p^*)$ then this form induces  a $G$-invariant form on $X$. By abuse of
  notation we denote the corresponding volume form also by $\omega$.

  \begin{thm}\label{invariance} The map $I:\mathbb C[X]\to \mathbb C$ defined by
    $$I(f)=\int_Xf\omega$$
    is $G$-equivariant.
  \end{thm}
  \begin{proof} If $\delta$ is any volume form on $X$, $Q$ is a vector field on $X$ then $\int_X Q(\delta)=0$ (see Witten \cite{witten} 3.53).
    For any $u\in\mathfrak g$ we denote by $L_u$ the corresponding vector field on $X$. Since $L_u(\omega)=0$
we have $\int_XL_u(f\omega)=\int_XL_u(f)\omega=0$ which implies $\mathfrak g$-invariance of $I$. Since $G$ is connected the statement follows.
  \end{proof}
  We call the value $I(1)$ the CS volume of $X$. Our goal is to check when CS volume of $X$ is not zero for some examples.

  The following lemma will help us to use Theorem \ref{thm:main_precise}. We call $u\in\mathfrak k_1$ {\it nondegenerate compact} if
  $[u,u]\in\operatorname{Lie}\uK_{\mathbb R}$ and $\operatorname{ad}_u$ induces an odd isomorphism of $\mathfrak p$. A nondegenerate compact $u$ is
  {\it regular} if $t:=u^2$ is a regular element of $\mathfrak g_0$.
  \begin{lm}\label{homog_fixed_points} If $u\in\mathfrak k_1$ is nondegenerate compact regular and $Q=L_u$ then $X^Q$ is finite and lies in
    $\uX_{\mathbb R}$.
    \end{lm}
    \begin{proof} Since $u\in\mathfrak k$,  $p\in X^Q$. Consider a point $x=gp$ with $g\in \uG$. Then $x\in X^Q$ iff $u$ lies in
      the stabilizer $\operatorname{Lie}G_x=\operatorname{Ad}_g(\mathfrak k)$, or equivalently, $\operatorname{Ad}^{-1}_gu\in \mathfrak k$.
      A semisimple element $t=u^2$ also acts like an isomorphism on $\mathfrak p$ and hence its centralizer is contained in $\mathfrak k$.
      Hence there exists a Cartan subalgebra $\mathfrak h$ of $\mathfrak g_0$ such that $t\in\mathfrak h\subset\mathfrak k_0$.
      Suppose for some $g\in\uG$ we have $\operatorname{Ad}^{-1}_gt\in \mathfrak k$. Since all Cartan subalgebras of $\mathfrak k_0$ are conjugate
      by $\uK$ we may assume without loss of generality that $\operatorname{Ad}^{-1}_gt\in \mathfrak h$. By regularity of $t$ we have
      $\operatorname{Ad}^{-1}_g(\mathfrak h)=\mathfrak h$, and thus we may further assume that $g\in WH$, where $W$ is in the Weyl group of $\uG$ and
      $H$ the maximal torus with Lie algebra $\mathfrak h$. Finiteness of $W$ implies finiteness of $X^Q$.
      Moreover, $W=N_{G_\rr}(H_\rr)/H_{\mathbb R}$ and hence we can choose $g\in\uG_{\mathbb R}$. This implies $X^Q\subset \uX_{\mathbb R}$.
    \end{proof}
    To compute $\int_X\omega$ we now use Theorem \ref{thm:main_precise} with $Q=L_u$ for some nondegenerate compact $u\in\mathfrak k_1$.
    If $x\in X^Q$ we denote by $\Loc^u_x$ the local contribution $\Loc_{T^*_xX}$ for $Q=L_u$. If $x=gp$ then we have the relation
    $$\Loc_{x}^u=\Loc_{p}^{\operatorname{Ad}_g^{-1}u}.$$
    Note that $u'={\operatorname{Ad}_g^{-1}u}$ is again regular. Since $\operatorname{sdim}\mathfrak p=0$ that implies that $u'$ is non-degnerate.
    Thus, we obtain
 \begin{equation}\label{homform}
      \int_{X}\omega=\sum_{gp\in X^{Q}}\Loc_{p}^{\operatorname{Ad}_g^{-1}u}.
      \end{equation}
      \begin{lm}\label{help} Let $\mathfrak s$ be a subalgebra of
        $\mathfrak k$. Assume that $\mathfrak p=V\oplus V^*$ as an $\mathfrak s$-module and this decomposition and a volume form $\omega$ satisfy
        the conditions of Lemma
        \ref{canonical_form}. Then for any nondegenerate regular compact $u\in\mathfrak s_1$ we have
        $$\Loc_{p}^u\omega= \left(\frac{2\pi}{\mathbf i}\right)^n.$$ In particular, if $\mathfrak s=\mathfrak k$
        $$\int_X\omega= \left(\frac{2\pi}{\mathbf i}\right)^n|X^Q|\neq 0.$$
      \end{lm}
      \begin{proof} The first assertion follows immediately from Lemma \ref{canonical_form}. The second is a consequence of the first and (\ref{homform}).
        \end{proof}

  Let $\sigma$ be some
  involutive automorphism
  of $G$. Then the scheme $K=G^\sigma$ of $\sigma$ fixed points is again quasireductive.
  The underlying manifold $\uX=\uG/\uK$ is a complex symmetric space. One can choose a $\sigma$-stable compact real form $\uG_{\mathbb R}\subset \uG$.
  Then $\uX_{\mathbb R}=\uG_{\mathbb R}/\uG_{\mathbb R}^{\sigma}$ is a compact symmetric space.
   It is obvious that $\uX_{\mathbb R}$ is totally real and therefore $X$ is a CS manifold.
 
    \subsection{Isotropic grassmannian}
    
    In this subsection $G=SOSp(2n|2n)$ with $n\geq 1$ and $K=GL(n|n)$. Let $E$ denote the standard $GL(n|n)$-module, we fix a symmetric bilinear
    $K$-invariant form on
    $E\oplus E^*$. Then $G$ is identified with connected component of the group of automorphisms preserving this form.  In this case
\begin{equation}\label{decomp}
  \mathfrak p\simeq \Lambda^2E\oplus\Lambda^2E^*
\end{equation}
  and the form $\omega$ is defined as in Lemma \ref{canonical_form}.
    Finally $\uG_{\mathbb R}\simeq SO(2n)\times SP(n)$ (for $SP(n)$ the compact form of $SP(2n,\mathbb C)$) and $\uX_{\mathbb R}$ is a direct product of a
    connected component of the grassmannian of maximal isotropic subspaces in $E_0\oplus E_0^*$ and the grassmannian of Lagrangian subspaces in
    $E_1\oplus E_1^*$. 
    \begin{thm}\label{isotropic} The CS volume of $X$ is not zero. More precisely
     $$\int_X\omega=2^{n-1}\left(\frac{2\pi}{\mathbf i}\right)^{2n^2}.$$ 
    \end{thm}
    \begin{proof} Let  $a_1,\dots a_n\in\mathbb R^*$ be such that $a_k\pm a_j\neq 0$ for all $k\neq j$. Let
      $u=\left(\begin{matrix}0&1_n\\ D&0\end{matrix}\right)$, where $D$ is the  diagonal $n\times n$-matrix with eigenvalues
      ${\mathbf i}a_n,\dots,{\mathbf i}a_n$. One can easily see that $u$ is a nondegenerate compact regular odd element of
      $\mathfrak k=\mathfrak{gl}(n|n)$. Note that (\ref{decomp}) ensures that we can use the second part of Lemma \ref{canonical_form}.
      Therefore the CS volume of $X$ is not zero.

      To obtain the precise formula we just have to show that   $|X^Q|=2^{n-1}$. Following the proof of Lemma \ref{help} we just have to find the
      elements of the Weyl group $W$ such that $w^{-1}(u)$ lies in $\mathfrak k$. Furthermore, $w$ and $w'$ give the same fixed point iff $w'\in wW_K$
      where $W_K$ is the Weyl group of $\uK$. In our case $W_K=S_n\times S_n$ and $W=W_K\ltimes B$ where $B$ is the normal subgroup isomorphic to
      $\mathbb Z_2^{2n-1}$. Let $\mathfrak h$ be the diagonal subalgebra of $\mathfrak{gl}(n|n)$, it is a Cartan subalgebra of both
      $\mathfrak g$ and $\mathfrak k$. Choose the natural basis $\varepsilon_1,\dots,\varepsilon_n,\delta_1,\dots,\delta_n$. The elements of $W_K$ permute
      separately $\varepsilon_i$-s and $\delta_j$-s, while element of $B$ change the signs of $\varepsilon_i$-s and $\delta_j$-s with one
      restriction that only even number of sign changes for $\varepsilon_i$-s is allowed.
      The odd roots of $\mathfrak k$ are $\pm(\varepsilon_i-\delta_j)$, for $1\leq i,j\leq n$. Set
      $$\alpha_i:=\varepsilon_i-\delta_i,\ i=1,\dots,n.$$
      Note that $u$ is a generic element in
      $$\bigoplus_{i=1}^n(\mathfrak{g}_{\alpha_i}\oplus\mathfrak{g}_{-\alpha_i}).$$
      Let $w\in B$ and $w(u)\in\mathfrak k$.\footnote{In this case $w=w^{-1}$.} Then $w(\alpha_i)$ is a root of $\mathfrak k$
        and hence $w(\alpha_i)=\pm\alpha_i$, i.e.
      either $w(\varepsilon_i)=-\varepsilon_i, w(\delta_i)=-\delta_i$ or $w(\varepsilon_i)=\varepsilon_i, w(\delta_i)=\delta_i$.
      Taking into account that $w$ changes signs of even number of $\varepsilon_i$-s, we have $2^{n-1}$ choices for $w$. 
      \end{proof}
      \subsection{Symmetric superspace $P(n)/P(r)\times P(s)$}
      Let $n\geq 2$ and $n=r+s$. Assume that $G$ is the periplectic supergroup $P(n)$ and $K=P(r)\times P(s)$. If $E(n)=\mathbb C^{n|n}$ is equipped with
      odd symmetric non-degenerate bilinear form $\beta$ then $G\subset GL(E(n))$ is the supergroup preserving $\beta$. The subgroup $K$ preserves
      $\beta$-orthogonal decomposition decomposition $E(n)=E(r)\oplus E(s)$. Furthermore $\mathfrak p\simeq \Pi E(r)\otimes E(s)$ as $K$-modules.
      Note that $E(r)\simeq \Pi E(r)^*$ and $E(s)\simeq \Pi E(s)^*$ and therefore $\mathfrak p$ is selfdual and $K$ preserves a volume form $\omega$
      on $\mathfrak p$.
      We have $\uG=GL(n)$, $\uK=GL(r)\times GL(s)$ and $\uX_{\mathbb R}=U(n)/U(r)\times U(s)$ is the grassmannian of $r$-dimensional subspaces in
      $\mathbb C^n$.

      Recall that $\mathfrak g$ can be described as the algebra of $2n\times 2n$-matrices of the form
      $\left(\begin{matrix}A&B&\\C&-A^t\end{matrix}\right)$ with $B^t=B$ and $C^t=-C$.
      \begin{thm}\label{periplectic} The CS volume of $G/K$ is zero if and only if $rs$ is odd.
      \end{thm}
      \begin{proof} We start with showing that CS volume is zero if $rs$ is odd. 
        Let $u=\left(\begin{matrix}0&0&\\C&0\end{matrix}\right)$ with some non-degenerate skew-symmetric $C$. Then the adjoint $\uG$-orbit of $u$
        does not intersect $\mathfrak k$, and therefore $X^Q=\emptyset$ for $Q=L_u$. Hence the volume is zero.

        Let us consider now the case of even $rs$. Without loss of generality we may assume that $r$ is even.
        We start with choosing nondegenerate compact regular $u\in\mathfrak k_1$. It is convenient to use the root decomposition of $\mathfrak g$.
        Recall that $\mathfrak g_0\simeq\mathfrak{gl}(n)$, we can choose the subalgebra $\mathfrak h$ of the diagonal matrices as a Cartan subalgebra.
        Let $\{\varepsilon_1,\dots,\varepsilon_n\}$ be the standard basis of $\mathfrak h^*$. Then
        $$\Delta_0=\{\varepsilon_i-\varepsilon_j\mid i\neq j\}$$
        is the set of even roots. The set of odd roots is
        $$\Delta_1=\{\varepsilon_i+\varepsilon_j\mid i\neq j\}\cup\{-\varepsilon_i-\varepsilon_j\mid i\neq j\}\cup\{2\varepsilon_i\}.$$
        Let $l=\lfloor\frac {n}{2}\rfloor$ and
        $$\alpha_1:=\varepsilon_1+\varepsilon_2,\dots,\alpha_l:=\varepsilon_{2l-1}+\varepsilon_{2l}.$$
        Consider the set $a_1,\dots a_l$ of distinct positive real numbers.
        Choose $u_i\in \mathfrak{g}_{\alpha_i}\oplus \mathfrak{g}_{-\alpha_l}$ such that $u_i^2=\frac{1}{2}[u_i,u_i]=\mathbf{i}a_ih_i$,
        where $h_i=(\varepsilon_{2i-1}-\varepsilon_{2i})^\vee$ denotes the coroot in $\mathfrak g_0$.\footnote{This choice is unique up to the
          adjoint action of the maximal torus of $\uG$.} One can see that $[u_i,u_j]=0$ for $i\neq j$.
        Now we set $u:=u_1+\dots+u_l$. Then $u^2=h_1+\dots+h_l$ is a regular semisimple element in $\mathfrak g_0$. Furthermore, $u\in\mathfrak k_1$
        and $u^2\in \operatorname{Lie}\uK_{\mathbb R}$.
        
        Let $\mathfrak s$ be the superalgebra generated $\mathfrak g_{\pm\alpha_i}$ for $i=1,\dots,l$.
        One can see that $\mathfrak s$ is isomorphic to the direct sum of $l$ copies of $\mathfrak{sl}(1|1)$.
        The even part $\mathfrak s_0$ coincides with the center of $\mathfrak s$. As a module over $\mathfrak s$, $\mathfrak p$ has a decomposition
        $\bigoplus_{\chi\in S} \mathfrak p_{\chi}$ over some $S\subset\mathfrak s^*_0$. Note that  $0\notin S$ by nondegeneracy of $u$.
        We can choose a decomposition $S=S^+\cup S^-$ where $S^-=-S^+$. Set
        $$V:=\bigoplus_{\chi\in S^+} \mathfrak p_{\chi},\ V^*:=\bigoplus_{\chi\in S^-} \mathfrak p_{\chi}.$$
        Let $\omega$ denote the canonical volume form on $\mathfrak p$ associated with this decomposition.
        As before we denote the induced invariant volume form on $G/K$ by the same letter.
        Note that $\dim\uX=2rs$. By Lemma \ref{help} we have
        $$\Loc_{p}^u\omega= \left(\frac{2\pi}{\mathbf i}\right)^{rs}.$$

        \begin{lm}\label{fixed_periplectic} We have $|X^Q|=\binom{l}{r/2}$ and
          $$\int_X\omega=\left(\frac{2\pi}{\mathbf i}\right)^{rs}\binom{l}{r/2}.$$ 
          \end{lm}
          \begin{proof} We compute $|x^Q|$ in the same way as for $SOSP(2n|2n)/GL(n|n)$.
            In this case $W=S_n$ and $W_K=S_r\times S_s\subset S_n$. The condition $w^{-1}(u)\in\mathfrak k$ implies that $w^{-1}(\alpha_i)$ is a root
            of $\mathfrak k$
            for all $i=1,\dots,l$. If we think about $w\in W$ as a permutation of $\{1,\dots,n\}$ then
            the latter condition is equivalent to the following:
for every odd $j\leq 2l$ either $w^{-1}(j),w^{-1}(j+1)\leq r$ or $w^{-1}(j),w^{-1}(j+1)>r$.

            By applying a suitable $w'\in W_K$ we can get an element $\tau=ww'\in W$ such that $\tau^{-1}(\alpha_i)=(\alpha_j)$ for all $i=1,\dots l$.
            Note that $\tau^{-1}(\mathfrak s)=\mathfrak s$ and therefore
            $$\Loc_{p}^{\tau^{-1}(u)}\omega= \left(\frac{2\pi}{\mathbf i}\right)^{rs}.$$
            The total number of such $\tau$ is $l!$, and the number of such $\tau$ modulo $W_K$ equals $\binom{l}{r/2}$. Therefore in this
            case formula (\ref{homform})
            gives the desired answer.
          \end{proof}
          Theorem follows.
          \end{proof}
          \subsection{Partial flags}
          Let us now assume that $\mathfrak g$ is a finite dimensional Kac-Moody superalgebra. Then $\mathfrak g_0$ is reductive, $\mathfrak g$ admits
          a non-degenerate invariant symmetric bilinear even  form. Furthermore, $\mathfrak g$ has a root decomposition
          $$\mathfrak g=\mathfrak h\oplus\bigoplus_{\alpha\in\Delta}\mathfrak g_{\alpha},$$
          with purely even $\mathfrak h$ and $\dim \mathfrak g_{\alpha}=(1|0)$ or $(0|1)$. The defect $d$ of $\mathfrak g$ is the maximal number
          of linearly independent mutually
          orthogonal isotropic roots. If $\mathfrak g$ has defect zero then it is a direct sum of finite-dimensional simple Lie algebras and
          $\mathfrak{osp}(1|2n)$ for different $n$. The category of finite-dimensional representations of an algebraic supergroup $G$ with
          $\mathfrak g=\operatorname{Lie}G$ is semisimple iff the defect of $\mathfrak g$ is zero.

          Let $\alpha_1,\dots,\alpha_d$ be a set of mutually orthogonal linearly isotropic roots and let $\mathfrak d$ denote the subalgebra generated
          by $\mathfrak g_{\pm \alpha_i}$ for $i=1,\dots,d$. The corresponding algebraic supergroup $D$ isomorphic to $SL(1|1)^d$ is called the
          defect subgroup of $G$, see Section 3.2 in \cite{serganova-sherman}. Let $\mathfrak k$ be the centralizer of $\mathfrak d_0$ and
          $\mathfrak c$ the centralizer of $\mathfrak d$. Then $\mathfrak k$ and $\mathfrak c$
          are quasireductive and
          $$\mathfrak k=\mathfrak c\oplus \tilde{\mathfrak d},$$
          where $\tilde{\mathfrak d}$ is the centralizer of $\mathfrak c$.
          Furthermore, it is not hard to see that $\mathfrak c$ is a Kac-Moody superalgebra of defect $0$ and $\tilde{\mathfrak d}$
          is isomorphic to $\mathfrak{gl}(1|1)^d$, in particular, $\mathfrak d=[\tilde{\mathfrak d},\tilde{\mathfrak d}]$.

          Let $G$ be a connected algebraic group with Lie superalgebra $\mathfrak g$ and $K$ the connected algebraic subgroup with Lie
          superalgebra $\mathfrak k$. Consider the homogeneous superspace $X=G/K$. Since $\mathfrak k$ is the centralizer of a generic element in
          $\mathfrak d_0\subset\mathfrak h$,  $\mathfrak k$ is a Levi superalgebra of some parabolic subalgebra
          $\mathfrak k\oplus \mathfrak u$ of $\mathfrak g$, here $\mathfrak u$ is the nilpotent radical.
          Thus, we have
          $\mathfrak g=\mathfrak k\oplus\mathfrak p$, and $\mathfrak p=\mathfrak u\oplus\mathfrak u^*$ as a $\mathfrak k$-module.
          This decomposition defines a canonical volume form $\omega$ on $\mathfrak p$ which extends to invariant volume form on $X$.
          The underlying compact real manifold $\uX_{\mathbb R}=\uG_{\mathbb R}/\uK_{\mathbb R}$ is a partial flag variety for $\uG$.
          \begin{thm}\label{main_hom} Let $W_d$ be the subgroup of the Weyl group $W$ which preserves the set $\{\pm\alpha_1,\dots,\pm\alpha_d\}$
            and $W_c$ the Weyl group of $\mathfrak c$.
            Then
            $$\int_X\omega=\frac{|W_d|}{|W_c|} \left(\frac{2\pi}{\mathbf i}\right)^{\dim_{\mathbb C}\uX}.$$
          \end{thm}
          \begin{proof} We choose $t\in(\mathfrak d_0)_{\mathbb R}$ such that the centralizer of $t$ coincides with $\mathfrak k$, and
            then $u\in\mathfrak d_1$ such that $t=u^2$. Then $\operatorname{ad}_u:\mathfrak p\to\mathfrak p$ is an isomorphism.
            Let $Q=L_u$. Let us compute $X^Q$. If $gp\in X^Q$ then $\operatorname{Ad}_g^{-1}(u)\in \mathfrak k_1$. Since
            $\mathfrak k_1=\mathfrak c_1\oplus\mathfrak d_1$ and $\mathfrak c$ has zero defect, $\operatorname{Ad}_g^{-1}(u)\in \mathfrak d_1$.
            That implies $g$ lies in the normalizer of $\mathfrak d$ and hence in the normalizer of $\mathfrak c\oplus\mathfrak d$.
            Finally, $\mathfrak k$ is the normalizer of $\mathfrak c\oplus\mathfrak d$ in $\mathfrak g$ and hence we have
            $\operatorname{Ad}_g^{-1}(\mathfrak k)=\mathfrak k$. On the other hand,
            if $N_{\uG}(\mathfrak k)$ denotes the normalizer of $\mathfrak k$ then
            $$N_{\uG}(\mathfrak k)/\uK\simeq W_d/W_c.$$
            Now the statement follows from the second part of Lemma \ref{help}.
            \end{proof}
            \subsection{Application to splitting subgroups}
            Let $G$ and $K\subset G$ be quasireductive algebraic groups and $X=G/K$. The space $\mathbb C[X]$ of regular functions on
            $X$ is naturally a $G$-module. Recall that $K$ is called {\it splitting} in $G$ if the trivial submodule $\mathbb C\subset\mathbb C[X]$
            splits as a direct summand in $\mathbb C[X]$.  We refer the reader to \cite{serganova-sherman} for detailed explanation of importance
            of splitting subgroups. We summarize some properties of splitting subgroups in the following 
            \begin{prop}\label{splitting} Let $K\subset G$ be quasireductive algebraic groups and $\operatorname{Rep}G$, $\operatorname{Rep}K$
              denote the categories of representations of $G$ and $K$ respectively.

              (a) $K$ is splitting in $G$ iff for any $G$-modules $M$ and $M'$ the restriction map
              $$\operatorname{Ext}^i_G(M,M')\to \operatorname{Ext}^i_K(M,M')$$
              is injective for all $i$.

              (b)  Let $K$ be splitting in $G$ and $M$ be a $G$-module. If $\operatorname{Res}_KM$ is projective in $\operatorname{Rep}K$ then
              $M$ is projective in $\operatorname{Rep}G$.

              (c) If $H\subset K\subset G$, $H$ is splitting in $K$ and $K$ is splitting in $G$ then $H$ is splitting in $G$. 
            \end{prop}
            Now we relate the property of being splitting with CS volume of $X$.
            \begin{prop}\label{volume_splitting}   Assume that
              $\uX_{\mathbb R}=\uG_{\mathbb R}/(\uK\cap\uG_{\mathbb R})$ is a totally real submanifold of $\uX$ and that $X$ admits a $G$-invariant
              volume form $\omega$.
              If CS volume of $X$ is not zero then $K$ is splitting in $G$.
            \end{prop}
            \begin{proof} Follows immediately from Theorem \ref{invariance} since the composition
              $$\mathbb C\to\mathbb C[X]\xrightarrow{I}\mathbb C$$
            is not zero.
          \end{proof}

          \begin{thm}\label{defect} Let $\mathfrak g=\operatorname{Lie} G$ be a Kac-Moody Lie superalgebra and $\mathfrak d$ be
            its defect subalgebra. Then the corresponding defect subgroup $D$
            is splitting in $G$.
          \end{thm}
          \begin{proof} Let $K$ be the connected component of the centralizer of $\mathfrak d_0$. By Theorem \ref{main_hom} $K$ is splitting in $G$.
            Note that $D$ is normal in $K$ and the category
            $\operatorname{Rep} K/D$ is semisimple. Therefore $D$ is splitting in $K$. Thus, by Proposition \ref{splitting} (c)
            $D$ is splitting in $G$.            
          \end{proof}
          \begin{rmk}\label{tilde} Note that by the same argument $\tilde D$ is splitting in $G$.
          \end{rmk}

          \begin{prop}\label{p-splitting} Let $G=P(2l)$ (respectively, $P(2l+1)$) and $K=P(2)^l$ (respectively, $P(2)^l\times P(1)$) embedded diagonally
            into $G$. Then $K$ is splitting in $G$.
          \end{prop}
          \begin{proof} Theorem \ref{periplectic} and Proposition \ref{volume_splitting} imply that $P(2)\times P(n-2)$ is splitting in $P(n)$.
            Now the statement easily follows by induction on $n$ using Proposition \ref{splitting}.
            \end{proof}
          
          \subsection{ Associated variety  and projectivity criterion}
          Let $\mathfrak g$ be a finite-dimensional Lie superalgebra. The self-commuting cone $\mathcal X$ of $\mathfrak g$ is defined by
          $$\mathcal X=\{u\in\mathfrak g\mid [u,u]=0\}.$$
          If $u\in\mathcal X$ and $M$ is a $\mathfrak g$-module the induced linear map $u_M:M\to M$ satisfies the condition $u_M^2=0$. We define
          $$DS_uM:=\operatorname{Ker}u_M/\operatorname{Im}u_M.$$
          In fact, $DS_u$ defines a symmetric monoidal functor from the category of $\mathfrak g$-modules to the category of vector superspaces.
          This functor has many applications in representation theory of superalgebra, see \cite{GHSS}.
          The associated variety $\mathcal X_M$ of $M$ is defined by
          $$\mathcal X_M=\{u\in \mathcal X\mid DS_uM\neq 0\}.$$
          If $M$ is finite-dimensional then $\mathcal X_M$ is a closed  subvariety of $\mathcal X$.

          Assume now that $\mathfrak g$ is the Lie superalgebra of a quasireductive algebraic supergroup $G$. We denote by $\mathcal F(\mathfrak g)$
          the category of finite-dimensional $\mathfrak g$-modules semisimple over $\mathfrak g_0$. This is a Frobenius rigid symmetric monoidal category,
          in particular, it has enough projective objects. If $G$ is connected then the category $\operatorname{Rep}^fG$ of finite-dimensional
          $G$-modules is a tensor subcategory of
          $\mathcal F(\mathfrak g)$.  If $M\in \mathcal F(\mathfrak g)$ then $M\in \operatorname{Rep}G$ if and only if all weights of
          $M$ lie in the weight lattice of $\uG$.
          If $M$ is projective in $\mathcal F(\mathfrak g)$ then $\mathcal X_M=\{0\}$, see Theorem 10.2 in \cite {GHSS}.
          The converse is proven in \cite {GHSS} for the type I superalgebras, see Theorem 10.4. We are now able to prove it for Kac-Moody superalgebras.
          \begin{thm}\label{projectivity} Let $\mathfrak g$ be a Kac-Moody superalgebra. Then $M\in\mathcal F(\mathfrak g)$ is projective if and only if
            $\mathcal X_M=\{0\}$. 
          \end{thm}
          \begin{proof} We only have to show that if $\mathcal X_M=\{0\}$ then $M$ is projective. Let us assume first that $M\in \operatorname{Rep} G$.
            We have $\mathcal X_M\cap\tilde{\mathfrak d}_1=\{0\}$. By Theorem 10.4 in \cite{GHSS} $M$ is a projective $\tilde{\mathfrak d}$-module.
            Since $\tilde D$ is splitting in $G$, see Remark \ref{tilde}, $M$ is a projective $G$-module by Proposition \ref{splitting}(b).

            Now let us show the same for any $M\in\mathcal F(\mathfrak g)$. Without loss of generality we may assume that $M$ is
            indecomposable. Then weights of
            $M\otimes M^*$ lie in the root lattice and hence in the weight lattice of $\uG$. Therefore $M\otimes M^*\in \operatorname{Rep}G$.
            On the other hand, $\mathcal X_{M\otimes M^*}=\{0\}$. Therefore $M\otimes M^*$ is projective. But then $M$ is projective because
            $M$ is a direct summand of $M\otimes M^*\otimes M$.
            
            \end{proof}

\end{document}